\documentclass[journal]{IEEEtran}                  

\usepackage{amsmath,amsfonts,amssymb,amsthm}
\usepackage{mathrsfs}
\usepackage{mathtools}
\usepackage{array}
\usepackage{float}
\usepackage{extarrows}
\usepackage{cite}
\usepackage{url}
\usepackage{color}
\usepackage[table]{xcolor}
\usepackage[colorlinks,linkcolor=orange,citecolor=blue]{hyperref}
\usepackage{algorithmic}
\usepackage{graphicx}
\usepackage{textcomp}
\usepackage[font=footnotesize]{caption}
\usepackage[font=footnotesize]{subcaption}
\usepackage{stfloats}
\usepackage{epstopdf}
\usepackage{arydshln}
\usepackage{ifthen}
\usepackage{pict2e}
\usepackage[T1]{fontenc}
\usepackage{enumitem}
\usepackage{tikz}

\allowdisplaybreaks[1] 

\DeclareGraphicsExtensions{.pdf,.png}

\DeclareMathOperator{\diag}{diag}

\DeclareMathOperator{\trace}{trace}

\DeclareMathOperator{\vect}{vec}
\DeclareMathOperator{\range}{range}

\definecolor{myblue}{RGB}{231, 245, 254}

\newtheorem{theorem}{Theorem}

\newcommand{\dd}{\mathrm{d}} 
\renewcommand{\t}{^{\mbox{\tiny\sf T}}} 

\newcommand{\R}{\mathbb{R}}

\def\send#1#2{\stackrel{#1}{\hbox to #2{\rightarrowfill}}}
\def\-{\!\!\!\!\!-}

\newcommand{\rank}{{\rm rank\;}}

\newtheorem{lemma}{Lemma}
\newtheorem{remark}{Remark}
\newtheorem{proposition}{Proposition}
\newtheorem{corollary}{Corollary}

\def\R{{\rm I\!R}} 

\newcounter{seqn}[equation]
\def\theseqn{\arabic{equation}\alph{seqn}}

\def\endseqn{\eqno \@seqnnum
$$\ignorespaces}
\def\@seqnnum{(\theseqn)}

\newskip\mcentering \mcentering=0pt plus 1000pt minus 1000pt

\def\meqalignno#1{
\halign to\displaywidth{
    \hbox to 0pt{\kern\displaywidth\llap{$##$}\hss}\tabskip=\mcentering
    &\hfil$\displaystyle{##}$\tabskip=\mcentering
   &&$\displaystyle{{}##}$\hfil\tabskip=\mcentering
    \crcr
    #1\crcr}}

\def\dspace{\multiply\normalbaselineskip 150
		  \divide\normalbaselineskip 100 \normalbaselines
		  \csname @@normalbaselineskip\endcsname\normalbaselineskip}
\def\sspace{\multiply\normalbaselineskip 200
		 \divide\normalbaselineskip 300 \normalbaselines
		 \csname @@normalbaselineskip\endcsname\normalbaselineskip}
\def\sdspace{\multiply\normalbaselineskip 160
		 \divide\normalbaselineskip 150 \normalbaselines
		 \csname @@normalbaselineskip\endcsname\normalbaselineskip}

\def\@{\tilde}

\def\3dot#1{\buildrel\textstyle...\over#1}

\renewcommand{\R}{\mathbb{R}}

\begin{document}

\title{Optimal Covariance Steering for \\
Continuous-Time Linear Stochastic Systems \\
With Multiplicative Noise}

\author{Fengjiao Liu and Panagiotis Tsiotras
\thanks{This work has been supported by NASA University Leadership Initiative award 80NSSC20M0163 and ONR award N00014-18-1-2828.}
\thanks{F. Liu and P. Tsiotras are with the School of Aerospace Engineering, Georgia Institute of Technology, Atlanta, GA 30332 USA (e-mail: \{fengjiao, tsiotras\}@gatech.edu).}}

\maketitle

\begin{abstract}
In this paper we study the finite-horizon optimal covariance steering problem for a continuous-time linear stochastic system subject to both additive and multiplicative noise.
The noise can be continuous or it may contain jumps. 
Additive noise does not depend on the state or the control, whereas multiplicative noise has a magnitude proportional to the current state. 
The cost is assumed to be quadratic in both the state and the control. 
First, the controllability of the state covariance is established under mild assumptions. 
Then, the optimal control for steering the covariance from some initial to some final value is provided. 
Lastly, the existence and uniqueness of the optimal control is shown. 
In the process, we provide a result of independent interest regarding the maximal interval of existence of the solution to a matrix Riccati differential equation.
\end{abstract}


\section{Introduction}

Covariance control theory aims to quantify and control the uncertainty in dynamical systems. 
For a brief history of covariance control theory, please refer to \cite[Section I]{liu2022add}. 
Covariance control, specifically over a finite horizon, is often referred to as covariance steering. 
In general, there are infinitely many ways to steer the state covariance from a given initial covariance to a given final covariance. 
We are particularly interested in optimally steering the state covariance of a continuous-time linear stochastic system, with respect to a quadratic cost functional
\begin{equation} \label{cost-Q}
J(u) \triangleq \mathbb{E}\left[\int_{0}^{1} \Big( x(t)\t Q(t) x(t) + u(t)\t R(t) u(t) \Big) \, \dd t\right],
\end{equation}
where $x(t) \in \mathbb{R}^{n}$ is the state, $u(t) \in \mathbb{R}^{p}$ is the control input, and $R(t) \succ 0$ and $Q(t) \succeq 0$ are matrices of dimensions $p \times p$ and $n \times n$, respectively, and are continuous on the time interval $[0, 1]$.

Problems with cost such as \eqref{cost-Q} have been studied for linear stochastic systems subject to additive white Gaussian noise in \cite{chen2016I, chen2016II, chen2018III}. 
Specifically, it is shown that there exists a unique optimal control for steering the state covariance from any initial positive definite covariance to any final positive definite covariance, and the optimal control can be solved in closed form, provided the noise channel coincides with the control channel. 
Subsequently, the authors of \cite{liu2022add} addressed linear stochastic systems corrupted by additive generic noise, which is modeled by the ``differential'' of a continuous-time martingale and may contain both white Gaussian noise and random jumps of any size. 
The authors of \cite{liu2022add} pointed out that when the noise channel is different from the control channel, and with $Q(t) \equiv 0$, there also exists a unique optimal control, although the optimal control may not admit a closed-form solution.


Thus far, all noise models in the finite-horizon covariance steering literature consider only additive noise that is independent of the state and control input. 
Nevertheless, in numerous engineering applications the system is also prone to state-dependent noise, control-dependent noise, or general multiplicative noise \cite{volpe2016effective}. 
Examples of systems corrupted by multiplicative noise include 
neuro-physiological systems \cite{harris1998signal}, 
asset pricing models \cite{bormetti2014multiplicative}, 
signal processing processes \cite{loo1968statistical}, 
aerospace systems \cite{mclane1971optimal}, 
and electromechanical systems \cite{wu2019preview}. 
Motivated by these applications, this paper aims to extend the results of \cite{liu2022add} to the case of linear stochastic systems subject to multiplicative noise, and also having a more general quadratic cost functional, with $Q(t) \succeq 0$.


\textit{Contributions:} 
This paper claims the following contributions. 
Firstly, we present a candidate optimal control law for covariance steering with respect to a cost functional which is quadratic in both the state and the control variables. 
Secondly, the controllability of the state covariance over a finite time interval is analyzed for a linear stochastic system subject to both additive and state-dependent noise. 
Finally, the existence and uniqueness of the optimal control is confirmed. 
Interestingly, we are also able to give a simple characterization on the existence of the solution to a matrix Riccati differential equation when the linear time-varying system without noise is totally controllable.


The rest of the paper is organized as follows. 
The main optimal covariance steering problem addressed in this paper is formulated in Section \ref{sec:prob}. 
A candidate optimal control law is derived in Section \ref{sec:cond}. 
The controllability of the state covariance is established in Section \ref{sec:contr} for a linear stochastic system subject to both additive and state-dependent noise. 
The existence and uniqueness of the optimal control is shown in Section \ref{sec:soltn-nd}. 
A numerical example is provided in Section \ref{sec:expl} to demonstrate the results of this paper. 
To keep the discussion brief, we refer the reader to \cite[Section II]{liu2022add} for the technical background on the generic noise model along with some preliminaries on stochastic processes and It\^{o} calculus.


\section{Problem Formulation} \label{sec:prob}

Consider the linear time-varying stochastic system, which is corrupted by a combination of additive and multiplicative noise,
\begin{multline} \label{sde:mult-jump}
\dd x(t) = A(t) x(t) \, \dd t + B(t) u(t) \, \dd t + C(t) \, \dd m(t) 
\\* 
+ \bigg(\sum_{i = 1}^{\ell} E_{i}(t) \, \dd \mu_{i}(t)\bigg) x(t),
\end{multline}
satisfying the following initial condition,
\begin{equation} \label{bgn:mean-cov-Q}
\mathbb{E}\left[x(0)\right] = 0, \quad \mathbb{E}\big[x(0) x(0)\t \big] = \Sigma_{0} \succ 0, 
\end{equation}
where $x(t) \in \mathbb{R}^{n}$ is the state vector at time $t$, $u(t) \in \mathbb{R}^{p}$ is the control input at time $t$, 
$m(t) \in \mathbb{R}^{q}$ and $\mu_{i}(t) \in \mathbb{R}$, where $i \in \{1, 2, \dots, \ell\}$, are independent square integrable martingales independent of the state history $\{x(s): \, 0 \leq s \leq t\}$ and with $m(0) = 0$, $\dd\mathbb{E}\left[ m(t)m(t)\t \right]/\dd t = D(t) \succeq 0$, and $\mu_{i}(0) = 0$, $\dd\mathbb{E}\left[\mu_{i}^{2}(t)\right]/\dd t = 2\nu_{i}(t) \geq 0$, respectively, 
and $A(t) \in \mathbb{R}^{n \times n}$, $B(t) \in \mathbb{R}^{n \times p}$, $C(t) \in \mathbb{R}^{n \times q}$, and $E_{i}(t) \in \mathbb{R}^{n \times n}$ are known coefficient matrices. 
Let $\mathcal{C}^{k}$ denote the class of $k$-times continuously differentiable functions defined on the interval $[0, 1]$. 
We assume that $A(t) \in \mathcal{C}^{n-1}$, $B(t) \in \mathcal{C}^{n}$, and $C(t)$, $D(t)$, $E_{i}(t)$, $\nu_{i}(t) \in \mathcal{C}^{0}$. 
Without loss of generality, 
we assume that \eqref{sde:mult-jump} is defined on the interval $[0, 1]$, and that the desired terminal state $x(1)$ is characterized by its mean and covariance matrix given by 
\begin{equation} \label{tgt:mean-cov-Q}
\mathbb{E}\left[x(1)\right]=0,
\quad
\mathbb{E}\big[x(1) x(1)\t \big]=\Sigma_{1} \succ 0.
\end{equation}
A control input $u$ is said to be \emph{admissible} if, for each $t \in [0, 1]$, 
it depends only on $t$ and on the past history of the state $\{x(s): \, 0 \leq s \leq t\}$, and satisfies $J(u) < \infty$, where $J(u)$ is given by \eqref{cost-Q}, 
such that \eqref{sde:mult-jump} with the initial condition \eqref{bgn:mean-cov-Q} has a strong solution \cite{protter2003stochastic}, and the desired terminal state mean and covariance given by \eqref{tgt:mean-cov-Q} are reached. 
Let $\mathcal{U}$ denote the set of admissible controls. 
The problem is to check whether $\mathcal{U}$ is nonempty and, if so, to determine the optimal control $u^{*} \in \mathcal{U}$ that minimizes the quadratic cost functional \eqref{cost-Q} subject to the initial and terminal state constraints \eqref{bgn:mean-cov-Q}, \eqref{tgt:mean-cov-Q}.


\section{Optimal Control of the State Covariance} \label{sec:cond}

In this section, a candidate optimal control is developed for the covariance steering problem formulated in Section \ref{sec:prob}.

Let $\Pi(t)$, $t \in [0, 1]$, be a differentiable function taking values
in the set of $n \times n$ symmetric matrices. 
Because of the prescribed boundary conditions \eqref{bgn:mean-cov-Q} and \eqref{tgt:mean-cov-Q}, the expected values $\mathbb{E}\big[ x(0)\t \Pi(0) x(0)\big]$ and $\mathbb{E}\big[ x(1)\t \Pi(1) x(1) \big]$ are independent of the control $u \in \mathcal{U}$. 
Hence, it follows from a similar derivation as in \cite[Section V.A]{liu2022add} that the cost functional \eqref{cost-Q} can be equivalently written as
\begin{align} \label{cost:pre-sqr-Q}
&\tilde{J}(u) 
= \mathbb{E}\left[\int_{0}^{1} \Big(u(t)\t R(t) u(t) + x(t)\t Q(t) x(t)\Big) \, \dd t\right] 
\nonumber \\* 
&\hspace{14mm} 
+ \mathbb{E}\big[x(1)\t \Pi(1) x(1)\big] - \mathbb{E}\big[x(0)\t \Pi(0) x(0)\big] \nonumber \\
&= \mathbb{E}\left[\int_{0}^{1} \Big(u\t R u + u\t B\t \Pi x + x\t \Pi B u\Big) \, \dd t\right]
\nonumber \\* 
&\hspace{4mm} 
+ \mathbb{E} \! \left[\int_{0}^{1} \hspace{-1mm} x\t \bigg( \! \dot{\Pi} + A\t \Pi + \Pi A + Q + 2\sum_{i = 1}^{\ell} \nu_{i} E_{i}\t \Pi E_{i} \! \bigg) x \, \dd t\right]
\nonumber \\* 
&\hspace{4mm}
+ \int_{0}^{1} \trace \left(\Pi C D C\t \right) \dd t,
\end{align}
where we have dropped the function arguments for notational simplicity, and
we have applied the It\^{o} calculus formula (see Section II in \cite{liu2022add}).

To this end, let $\Pi(t)$ satisfy the Riccati differential equation
\begin{multline} \label{ode:pi-Q}
\dot{\Pi} = - A(t)\t \Pi - \Pi A(t) + \Pi B(t) R^{-1}(t) B(t)\t \Pi - Q(t) 
\\* 
- 2\sum_{i = 1}^{\ell} \nu_{i}(t) \, E_{i}(t)\t \Pi E_{i}(t).
\end{multline}
It follows that the cost functional in \eqref{cost:pre-sqr-Q} becomes 
\begin{multline} \label{cost:sqr-Q}
\tilde{J}(u) 
= \mathbb{E}\left[\int_{0}^{1} \Big\| R^{\frac{1}{2}}(t) u(t) + R^{-\frac{1}{2}}(t) B(t)\t \Pi(t) x(t) \Big\|^{2} \dd t\right] 
\\* 
+ \int_{0}^{1} \trace \left(\Pi C D C\t \right) \dd t.
\end{multline}
Since the second term in \eqref{cost:sqr-Q} is independent of the control $u$, a candidate optimal control takes the form
\begin{equation} \label{ctrl:opt-Q}
u^{*}(t) = - R^{-1}(t) B(t)\t \Pi(t) x(t). 
\end{equation}
The corresponding candidate optimal process is 
\begin{multline} \label{sde:mult-jump-opt} 
\dd x^{*} = \Big(A(t) - B(t) R^{-1}(t) B(t)\t \Pi(t)\Big) x^{*} \, \dd t + C(t) \, \dd m 
\\* 
+ \bigg(\sum_{i = 1}^{\ell} E_{i}(t) \, \dd \mu_{i}(t)\bigg) x^{*}.
\end{multline}
Since the initial condition is $\mathbb{E}\left[x(0)\right] = 0$ and system \eqref{sde:mult-jump} is subject to the state feedback control \eqref{ctrl:opt-Q}, it follows that $\mathbb{E}\left[x(t)\right] = 0$ for $t \in [0, 1]$.

Accordingly, let $\Sigma(t) = \mathbb{E}\big[x^{*}(t) x^{*}(t)\t \big]$ be the covariance of $x^{*}(t)$. 
Then, $\Sigma(t)$ satisfies the Lyapunov differential equation 
\begin{align} \label{ode:sigma-Q}
\dot{\Sigma} &= \Big(A(t) - B(t) R^{-1}(t) B(t)\t \Pi(t)\Big) \Sigma + C(t) D(t) C(t)\t 
\nonumber \\* 
&\hspace{5mm} + \Sigma \Big(A(t) - B(t) R^{-1}(t) B(t)\t \Pi(t)\Big)\t 
\nonumber \\* 
&\hspace{5mm} + 2\sum_{i = 1}^{\ell} \nu_{i}(t) E_{i}(t) \Sigma E_{i}(t)\t,
\end{align}
with boundary conditions
\begin{equation} \label{bdr:sigma-Q}
\Sigma(0) = \Sigma_{0} \succ 0, \qquad \Sigma(1) = \Sigma_{1} \succ 0.
\end{equation}
Thus, we have obtained the result below.

\begin{theorem} \label{thm:suff-Q}
Assume $\Pi(t)$ and $\Sigma(t)$ satisfy equations \eqref{ode:pi-Q}, \eqref{ode:sigma-Q}, \eqref{bdr:sigma-Q} for $t \in [0, 1]$. 
Then, the state feedback control $u^{*}$ given by \eqref{ctrl:opt-Q} is optimal for system \eqref{sde:mult-jump} with respect to the cost functional \eqref{cost-Q}, subject to the boundary constraints \eqref{bgn:mean-cov-Q}, \eqref{tgt:mean-cov-Q}. 
The corresponding optimal process is given by \eqref{sde:mult-jump-opt}.
\end{theorem}


In view of \eqref{cost:sqr-Q}, if there exists a unique solution to the coupled matrix ordinary differential equations (ODEs) \eqref{ode:pi-Q}, \eqref{ode:sigma-Q}, \eqref{bdr:sigma-Q}, then, the optimal control $u^{*}$ given by \eqref{ctrl:opt-Q} is unique.

For the time being, it is difficult to analyze these coupled matrix ODEs. 
Thus, for the rest of the paper, we assume the simpler case of state-dependent noise by letting $E_{i}(t) \equiv I_{n}$ for all $i \in \{1, 2, \dots, \ell\}$. 
Then, the linear stochastic system \eqref{sde:mult-jump} becomes
\begin{equation} \label{sde:mult-jump-simp}
\dd x(t) = A(t) x(t) \, \dd t + B(t) u(t) \, \dd t + C(t) \, \dd m(t) + x(t) \, \dd \mu(t),
\end{equation}
where, 
\begin{equation*}
\mu(t) \triangleq \sum_{i = 1}^{\ell} \mu_{i}(t), ~~ 
\dd\mathbb{E}\left[\mu^{2}(t)\right]/\dd t = 2 \sum_{i = 1}^{\ell} \nu_{i}(t) \triangleq 2\nu(t) \geq 0.
\end{equation*}
Accordingly, the coupled matrix ODEs \eqref{ode:pi-Q} and \eqref{ode:sigma-Q} become
\begin{align}
\dot{\Pi} &= - A(t)\t \Pi - \Pi A(t) + \Pi B(t) R^{-1}(t) B(t)\t \Pi - Q(t) 
\nonumber \\* 
&\hspace{60mm} - 2\nu(t) \Pi, \label{ode:pi-Q-simp} \\
\dot{\Sigma} &= \Big(A(t) - B(t) R^{-1}(t) B(t)\t \Pi(t)\Big) \Sigma + C(t) D(t) C(t)\t
\nonumber \\* 
&\hspace{5mm} + \Sigma \Big(A(t) - B(t) R^{-1}(t) B(t)\t \Pi(t)\Big)\t + 2\nu(t) \Sigma. \label{ode:sigma-Q-simp}
\end{align}

It will be shown in Section~\ref{sec:soltn-nd} that the simplified coupled matrix ODEs
\eqref{bdr:sigma-Q}, \eqref{ode:pi-Q-simp}, \eqref{ode:sigma-Q-simp},  have a unique solution. 
Thus, the optimal control for system \eqref{sde:mult-jump-simp} is unique with respect to the cost \eqref{cost-Q} and is given by \eqref{ctrl:opt-Q}.


\section{Controllability of the State Covariance} \label{sec:contr}

In this section, we show that $\mathcal{U}$ is nonempty under some mild conditions. 
The state covariance of \eqref{sde:mult-jump-simp}, written explicitly as
\begin{equation*}
\Sigma(t) \triangleq \mathbb{E}\left[\Big(x(t)-\mathbb{E}[x(t)]\Big)\Big(x(t)-\mathbb{E}[x(t)]\Big)\t \right]
\end{equation*}
is said to be \emph{controllable} on the time interval $[0, 1]$ if, for any given $\Sigma_{0}, \Sigma_{1} \succ 0$, there exists an admissible control $u \in \mathcal{U}$ that steers the state covariance $\Sigma(t)$ from $\Sigma(0) = \Sigma_{0}$ to $\Sigma(1) = \Sigma_{1}$, while maintaining $\Sigma(t) \succ 0$ for all $t \in [0, 1]$.

Consider the state feedback control of the form 
\begin{equation} \label{eqn:stt-fb}
u(t) = K(t)x(t), \quad t \in [0, 1], 
\end{equation}
where $K(t)$ is bounded on the interval $[0, 1]$. 
In light of $\mathbb{E}\left[x(0)\right] = 0$ and \eqref{eqn:stt-fb}, it is clear that $\mathbb{E}\left[x(t)\right] = 0$ for all $t \in [0, 1]$. 
It follows using a similar derivation as in \cite{liu2022add}, that with the control \eqref{eqn:stt-fb} the state covariance $\Sigma(t) = \mathbb{E}\big[x(t) x(t)\t \big]$ satisfies
\begin{multline} \label{ode:sigma-stt-fb-Q}
\dot{\Sigma} = \Big(A(t) + B(t) K(t)\Big)\Sigma + \Sigma\Big(A(t) + B(t) K(t)\Big)\t 
\\* 
+ C(t) D(t) C(t)\t + 2\nu(t) \Sigma.
\end{multline}


\begin{lemma} \label{lem:sigma-pos}
Let $\Sigma_{0} \succ 0$ and $K(t)$ be given. 
Then, $\Sigma(t) \succ 0$ for all $t \in [0, 1]$, where $\Sigma(t)$ satisfies \eqref{ode:sigma-stt-fb-Q} with the initial condition $\Sigma(0) = \Sigma_{0}$.
\end{lemma}

\begin{proof} 
Let $\Phi_{K}(t, \tau)$ denote the state transition matrix of $A(t) + B(t) K(t) + \nu(t) I_{n}$. Then,
\begin{multline*}
\Sigma(t) = \Phi_{K}(t, 0) \Sigma_{0} \Phi_{K}(t, 0)\t 
\\* 
+ \int_{0}^{t} \Phi_{K}(t, \tau) C(\tau) D(\tau) C(\tau)\t \Phi_{K}(t, \tau)\t \, \dd \tau.
\end{multline*}
Since $\Phi_{K}(t, 0)$ is nonsingular, we have 
\begin{equation*}
\Phi_{K}(t, 0) \Sigma_{0} \Phi_{K}(t, 0)\t \succ 0. 
\end{equation*}
Thus, $\Sigma(t) \succ 0$ for all $t \in [0, 1]$. \hfill
\end{proof}


\subsection{Time-Invariant Case}

First, we assume that the matrix pair $(A, B)$ is time-invariant and controllable. 
For simplicity, let $C(t) D(t) C(t)\t = M(t) \succeq 0$.

\begin{theorem} \label{thm:contr-AB}
Let $(A, B)$ be controllable, and let $\Sigma_{0}, \Sigma_{1} \succ 0$. 
Let $M, \nu \in \mathcal{C}^{n-2}$ be such that $M(t) \succeq 0$ and $\nu(t) \geq 0$ on $[0, 1]$. 
Then, 
there exists $K \in \mathcal{C}^{0}$ such that the solution of the matrix differential equation
\begin{equation} \label{ode:sigma-ABK}
\dot{\Sigma} = \Big(A + B K(t)\Big) \Sigma + \Sigma \Big(A + B K(t)\Big)\t + M(t) + 2\nu(t) \Sigma, 
\end{equation}
satisfies $\Sigma(t) \succ 0$ on $[0, 1]$ with boundary conditions $\Sigma(0) = \Sigma_{0}$ and $\Sigma(1) = \Sigma_{1}$.
\end{theorem}


In view of Lemma~\ref{lem:sigma-pos}, and since $\Sigma(t) \succ 0$, if we define $U(t) = \Sigma(t) K(t)\t$ equation \eqref{ode:sigma-ABK} is a linear function in terms of $U(t)$ (see \eqref{ode:sigma-ABU} below) and then we can recover $K(t)$ from $U(t)$ by letting $K(t) = U(t)\t \Sigma^{-1}(t)$. 
Thus, Theorem \ref{thm:contr-AB} is a direct consequence of the following result whose proof can be found in the Appendix.

\begin{proposition} \label{prp:U}
Let $(A, B)$ be controllable, 
let $H \geq 0$ be a given integer, 
and let $\Sigma_{0}, \Sigma_{1} \succ 0$. 
Let also $M, \nu \in \mathcal{C}^{H + n - 1}$ be such that $M(t) \succeq 0$ and $\nu(t) \geq 0$ on $[0, 1]$, 
and let $U_{i}^{0}, U_{i}^{1} \in \mathbb{R}^{n \times p}$, where $i \in \{0, 1, \dots, H\}$. 
There exists $\mathbb{R}^{n \times p}$-valued $U \in \mathcal{C}^{H}$ with boundary conditions $U(0) = U_{0}^{0}$, $U(1) = U_{0}^{1}$, $\frac{\dd^{i} U}{\dd t^{i}}(0) = U_{i}^{0}$, and $\frac{\dd^{i} U}{\dd t^{i}}(1) = U_{i}^{1}$, where $i \in \{1, \dots, H\}$, 
such that the solution of the matrix differential equation
\begin{equation} \label{ode:sigma-ABU}
\dot{\Sigma} = A \Sigma + \Sigma A\t + B U(t)\t + U(t) B\t + M(t) + 2\nu(t) \Sigma
\end{equation}
satisfies $\Sigma(t) \succ 0$ on $[0, 1]$ with boundary conditions $\Sigma(0) = \Sigma_{0}$ and $\Sigma(1) = \Sigma_{1}$.
\end{proposition}


With $(A, B)$ in the canonical form $(A_{n}, B_{n})$ (see \eqref{eqn:AB-par} in the proof of Proposition \ref{prp:U}), we can partition the matrix $\Sigma$ into $n$ layers labeled $L_{1}$, $L_{2}$, $\dots$, $L_{n}$, respectively, as in Figure~\ref{fig:par}, where the first layer $L_{1}$ is just the upper left entry of $\Sigma$ and the $n$th layer $L_{n}$ is the outmost layer consisting of the last column and last row of $\Sigma$. In each layer $k \in \{1, 2, \dots, n\}$, let $L_{k}^{\dagger}$ denote the first $k-1$ rows of $L_{k}$ and let $L_{k}^{\star}$ denote the lower right ``corner'' entry of $L_{k}$. 
From \eqref{eqn:Sig-M-U-par}-\eqref{ode:Sig-lr} in the proof of Proposition~\ref{prp:U}, we can conclude by induction from the outer layers to the inner layers that $L_{k}$ is controlled directly by $L_{k+1}^{\dagger}$ for $k \in \{1, 2, \dots, n-1\}$ and $L_{n}$ is controlled directly by the control input $U$.
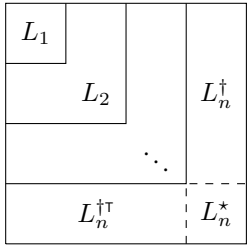
\begin{figure}[H] 
\begin{center}
\begin{tikzpicture}
\draw (0,2.4) -- (0.8,2.4) -- (0.8,3.2);
\node[rectangle] at (0.4,2.8) {$L_{1}$};
\draw (0,1.6) -- (1.6,1.6) -- (1.6,3.2);
\node[rectangle] at (1.2,2) {$L_{2}$};
\draw (0,0.8) -- (2.4,0.8) -- (2.4,3.2);
\node[rectangle] at (2,1.2) {$\ddots$};
\draw (0,0) -- (3.2,0) -- (3.2,3.2) -- (0,3.2) -- (0,0);
\draw[dashed] (2.4,0) -- (2.4,0.8) -- (3.2,0.8);
\node[rectangle] at (2.8,0.4) {$L_{n}^{\star}$};
\node[rectangle] at (2.8,2) {$L_{n}^{\dagger}$};
\node[rectangle] at (1.2,0.4) {$L_{n}^{\dagger \mbox{\tiny\sf T}}$};
\end{tikzpicture}
\end{center}
\caption{Partition of $\Sigma$ into $n$ layers.}
\label{fig:par}
\end{figure}

Finding $U$ involves two steps. 
The first step is to propagate the boundary conditions of $\Sigma$ from the outer layers to the inner layers, and within each layer from the bottom entries to the top entries. 
For example, if the control $U$ needs to satisfy the boundary conditions up to $H$th order derivative, each layer $k \in \{2, \dots, n\}$ will have to satisfy the boundary conditions up to $(H+n+1-k)$th order derivative. 
Within each layer $k$, the boundary conditions for $L_{k}^{\star}$ will be determined first, then followed by the boundary conditions for $L_{k}^{\dagger}$. 
The second step is to propagate the entries of $\Sigma(t)$ for $t \in [0, 1]$ from the inner layers to the outer layers and finally to $U(t)$ on $[0, 1]$, and within each layer from the top entries to the bottom entries. 
Taking layer $k \in \{3, \dots, n\}$ as an example, let $L_{k, i}$ denote the $i$th entry of $L_{k}$ from the top, where $i \in \{1, \dots, k\}$. 
Then, for each $j \in \{1, \dots, k-2\}$, $L_{k, j}(t)$ is determined by $L_{k-1, j}(t)$ and $L_{k-1, j+1}(t)$. 
Next, $L_{k, k-1}(t)$ is determined solely by $L_{k-1}^{\star}(t)$. 
Lastly, the corner entry $L_{k}^{\star}(t)$ is determined by $L_{1}(t)$, $\dots$, $L_{k-1}(t)$, and $L_{k}^{\dagger}(t)$. 
The boundary conditions of $U(t)$ are guaranteed by the first step.


\subsection{Time-Varying System}

Next, we assume that the matrix pair $\big(A(t), B(t)\big)$ is time-varying. 
Define 
\begin{align} \label{def:theta}
\Theta_{i}(t) &\triangleq 
\begin{bmatrix}
\Gamma_{0}(t) & \Gamma_{1}(t) & \cdots & \Gamma_{i-1}(t)
\end{bmatrix}
, ~ 1 \leq i \leq n+1, \\
\Gamma_{0}(t) &\triangleq B(t)
, \nonumber \\ 
\Gamma_{k}(t) &\triangleq -A(t)\Gamma_{k-1}(t) + \dot{\Gamma}_{k-1}(t)
, \quad 1 \leq k \leq n. \nonumber
\end{align}
The \emph{controllability matrix} of $\big(A(t), B(t)\big)$ is $\Theta_{n}(t)$ \cite{silverman1967controllability}. 
The pair $\big(A(t), B(t)\big)$ is \emph{totally controllable} on the time interval $[0, 1]$ if, for all $0 \leq t_{0} < t_{1} \leq 1$, there exists $t \in (t_{0}, t_{1})$ such that $\rank \Theta_{n}(t) = n$ \cite{stubberud1964controllability}. 
The pair $\big(A(t), B(t)\big)$ is \emph{uniformly controllable} on the time interval $[0, 1]$ if, for all $t \in [0, 1]$, $\rank \Theta_{n}(t) = n$ \cite{silverman1967controllability}. 
It follows immediately from these definitions that uniform controllability is stronger than total controllability. 
The pair $\big(A(t), B(t)\big)$ is \emph{index invariant} on the interval $[0, 1]$ if, for each $i \in \{1, 2, \dots, n+1\}$, $\rank \Theta_{i}(t)$ is constant for $t \in [0, 1]$, and $\rank \Theta_{n}(t) = \rank \Theta_{n+1}(t)$ \cite{morse1973structure}.


Now, assume that $\big(A(t), B(t)\big)$ is uniformly controllable and index invariant on $[0, 1]$. 
Following the same argument as in \cite{liu2022add}, we are able to reduce $\big(A(t), B(t)\big)$ to a time-invariant matrix pair via a time-varying coordinate transformation and a state feedback control. 
In light of Theorem \ref{thm:contr-AB}, we have reached the following result.
\begin{theorem} \label{thm:contr-Q}
Assume the matrix pair $\big(A(t), B(t)\big)$ is uniformly controllable and index invariant on the time interval $[0, 1]$. 
Then, the state covariance of the linear stochastic system \eqref{sde:mult-jump-simp} is controllable on $[0, 1]$ and also on any subinterval of $[0, 1]$. 
\end{theorem}


\begin{remark}
Theorem \ref{thm:exist-unique-Q} of Section \ref{sec:soltn-nd} states that if the pair $\big(A(t), B(t)\big)$ is totally controllable on $[0, 1]$, there exists a unique optimal control for system \eqref{sde:mult-jump-simp} for any $\Sigma_{0}, \Sigma_{1} \succ 0$. 
This fact immediately implies that Theorem \ref{thm:contr-Q} still holds when the assumptions of uniform controllability and invariant indices are relaxed to total controllability. 
However, a direct proof is not available at the moment. 
\end{remark}


\section{Solution to the Coupled ODEs} \label{sec:soltn-nd}

In this section, the existence and uniqueness of the solution to the coupled matrix ODEs \eqref{bdr:sigma-Q}, \eqref{ode:pi-Q-simp}, \eqref{ode:sigma-Q-simp}, is shown. 
We assume throughout this section that $\big(A(t), B(t)\big)$ is totally controllable on $[0, 1]$. 
That is, for all $0 \leq t_{0} < t_{1} \leq 1$, there exists $t \in (t_{0}, t_{1})$ such that $\rank \Theta_{n}(t) = n$, where $\Theta_{n}(t)$ is given by \eqref{def:theta}. 
Since $R(t) \succ 0$, one can easily check that $\big(A(t) + \nu(t) I_{n}, B(t)R^{-\frac{1}{2}}(t)\big)$ is also totally controllable on $[0, 1]$. 
For simplicity, it is assumed that $\nu(t) \equiv 0$ and that $R(t) \equiv I_{p}$ for the rest of this section, 
since we can define $A_{\text{new}}(t) = A(t) + \nu(t) I_{n}$ and $B_{\text{new}}(t) = B(t)R^{-\frac{1}{2}}(t)$ so that 
$A(t) = A_{\text{new}}(t) - \nu(t) I_{n}$ and $B(t) = B_{\text{new}}(t) R^{\frac{1}{2}}(t)$ can be recovered easily. 
Under these assumptions, a useful result is obtained from our analysis, which is a simple characterization of the maximal interval of existence of the solution to a matrix Riccati differential equation.


\subsection{Properties of the State Transition Matrix}

Let
\begin{equation} \label{eqn:phi-M-blk}
\Phi_{M}(t, s) \triangleq 
\begin{bmatrix}
\Phi_{11}(t, s) & \Phi_{12}(t, s) \\
\Phi_{21}(t, s) & \Phi_{22}(t, s)
\end{bmatrix}
\end{equation}
denote the state transition matrix for 
\begin{equation*}
M(t) \triangleq 
\begin{bmatrix}
A(t) & -B(t)B(t)\t \\
-Q(t) & -A(t)\t
\end{bmatrix}.
\end{equation*}
That is, $\Phi_{M}(t, s)$ satisfies 
\begin{equation*} 
\frac{\partial}{\partial t} \Phi_{M}(t, s) = M(t) \Phi_{M}(t, s), \quad \Phi_{M}(s, s) = I_{n}.
\end{equation*}
Let
\begin{equation*}
\Phi_{M}(t, s)^{-1} = \Phi_{M}(s, t) \triangleq 
\begin{bmatrix}
\Phi_{11}(s, t) & \Phi_{12}(s, t) \\
\Phi_{21}(s, t) & \Phi_{22}(s, t)
\end{bmatrix}.
\end{equation*}


\begin{lemma} \label{lem:phi-eqn}
For all $t, s \in \mathbb{R}$,
\begin{align} 
\Phi_{12}(t, s)\t \Phi_{22}(t, s) &= \Phi_{22}(t, s)\t \Phi_{12}(t, s), \nonumber \\
\Phi_{21}(t, s)\t \Phi_{11}(t, s) &= \Phi_{11}(t, s)\t \Phi_{21}(t, s), \nonumber \\
\Phi_{12}(t, s) \Phi_{11}(t, s)\t &= \Phi_{11}(t, s) \Phi_{12}(t, s)\t, \label{eqn:phi-blk-12} \\
\Phi_{21}(t, s) \Phi_{22}(t, s)\t &= \Phi_{22}(t, s) \Phi_{21}(t, s)\t, \label{eqn:phi-blk-21} \\
\Phi_{11}(t, s)\t \Phi_{22}(t, s) &- \Phi_{21}(t, s)\t \Phi_{12}(t, s) = I_{n}, \nonumber \\
\Phi_{11}(t, s) \Phi_{22}(t, s)\t &- \Phi_{12}(t, s) \Phi_{21}(t, s)\t = I_{n}. \label{eqn:phi-blk-11}
\end{align}
\end{lemma}

\begin{proof}
The first paragraph of the proof of Lemma 3 in \cite{chen2018III} implies that Lemma \ref{lem:phi-eqn} is true for all $t, s \in \mathbb{R}$. \hfill
\end{proof}


The proofs of the following four lemmas can be found in the Appendix.

\begin{lemma} \label{lem:phi-inv}
For all $t, s \in \mathbb{R}$, $\Phi_{11}(t, s)$ and $\Phi_{22}(t, s)$ are invertible. In particular,
\begin{align}
&\Phi_{11}(t, s) = \Phi_{22}(s, t)\t \label{eqn:phi-11-22} \\
&\hspace{10mm} = \Big(\Phi_{11}(s, t) - \Phi_{12}(s, t) \Phi_{22}(s, t)^{-1} \Phi_{21}(s, t)\Big)^{-1} \nonumber \\
&\hspace{10mm} = \Big(\Phi_{22}(t, s)\t - \Phi_{12}(t, s)\t \Phi_{11}(t, s)^{- \mbox{\tiny\sf T}} \Phi_{21}(t, s)\t \Big)^{-1}. \nonumber
\end{align}
\end{lemma}


\begin{lemma} \label{lem:phi-12-21}
For all $t, s \in \mathbb{R}$, 
\begin{align}
\Phi_{12}(t, s) &= - \Phi_{12}(s, t)\t, \label{eqn:phi-12-equiv} \\
\Phi_{21}(t, s) &= - \Phi_{21}(s, t)\t. \nonumber
\end{align}
\end{lemma}


\begin{remark}
It is worth pointing out that Lemma \ref{lem:phi-eqn}, Lemma \ref{lem:phi-inv}, and Lemma \ref{lem:phi-12-21} do not require $\big(A(t), B(t)\big)$ to be totally controllable, while the rest of the results in this section do. 
\end{remark}


\begin{lemma} \label{lem:phi-inv-mono}
For all $t \neq s$, $\Phi_{12}(t, s)$ is invertible. 
Moreover, for $t, s \in \mathbb{R}$,
\begin{equation} \label{eqn:phi-12}
- \Phi_{11}(t, s)^{-1} \Phi_{12}(t, s) = \Phi_{12}(s, t) \Phi_{22}(s, t)^{-1}.
\end{equation}
For $s < t_{1} < t_{2}$,
\begin{equation} \label{ineq:phi-mono-1}
0 \prec - \Phi_{11}(t_{1}, s)^{-1} \Phi_{12}(t_{1}, s) \prec - \Phi_{11}(t_{2}, s)^{-1} \Phi_{12}(t_{2}, s).
\end{equation}
For $t_{1} < t_{2} < s$,
\begin{equation} \label{ineq:phi-mono-2}
- \Phi_{11}(t_{1}, s)^{-1} \Phi_{12}(t_{1}, s) \prec - \Phi_{11}(t_{2}, s)^{-1} \Phi_{12}(t_{2}, s) \prec 0.
\end{equation}
\end{lemma}


\begin{remark}
A list of properties of the state transition matrix $\Phi_{M}$ defined in \eqref{eqn:phi-M-blk} is summarized in Table \ref{tbl:compare} in the Appendix.
\end{remark}


\subsection{Existence of Solution to the Riccati Differential Equation}

We show a necessary and sufficient condition for the solution $\Pi(t)$ of \eqref{ode:pi-Q-simp} to exist on $[0, 1]$, which leads naturally to the maximal interval of existence of a matrix Riccati differential equation.

\begin{lemma} \label{lem:pi-Q-exist-sol}
Let $\Pi(s)$ for some $s \in [0, 1]$ be given. 
Then, \eqref{ode:pi-Q-simp} admits a unique solution $\Pi(t)$ on $[0, 1]$ if and only if
\begin{equation} \label{cond:pi-Q-exist}
- \Phi_{12}(0, s)^{-1} \Phi_{11}(0, s) \prec \Pi(s) \prec - \Phi_{12}(1, s)^{-1} \Phi_{11}(1, s),
\end{equation}
where $\Phi_{12}(0, 0^{+})^{-1} = + \infty$ and $\Phi_{12}(1, 1^{-})^{-1} = - \infty$\footnote{Positive infinity of the $n \times n$ positive semidefinite cone, written $+ \infty$, is the limit of a sequence of $n \times n$ positive definite matrices whose eigenvalues all grow to $+ \infty$. Likewise, for $- \infty$. 
Notice that as $s \to 0^{+}$, all eigenvalues of $\Phi_{12}(0, s)$ go to $0^{+}$, and therefore all eigenvalues of $\Phi_{12}(0, s)^{-1}$ go to $+ \infty$. Likewise, as $s \to 1^{-}$, all eigenvalues of $\Phi_{12}(1, s)$ go to $0^{-}$, and therefore all eigenvalues of $\Phi_{12}(1, s)^{-1}$ go to $- \infty$.}. Moreover, 
\begin{multline} \label{sol:pi-Q}
\Pi(t) = \Big(\Phi_{21}(t, s) + \Phi_{22}(t, s) \Pi(s)\Big) 
\\* 
\times \Big(\Phi_{11}(t, s) + \Phi_{12}(t, s) \Pi(s)\Big)^{-1},
\end{multline}
and 
\begin{align} \label{ineq:pi-Q-bd}
- \Phi_{12}(0, t)^{-1} \Phi_{11}(0, t) 
&\prec \Pi(t) \nonumber \\
&\prec - \Phi_{12}(1, t)^{-1} \Phi_{11}(1, t), \quad t \in [0, 1].
\end{align}
\end{lemma}


\begin{corollary} \label{cor:pi-Q-exist}
Assume that, for all $t \in \mathbb{R}$, $\big(A(t), B(t)\big)$ is totally controllable and 
let $\mathcal{I}_{s} \subset \mathbb{R}$ be the maximal interval of existence of the solution to \eqref{ode:pi-Q-simp}, starting from $\Pi(s) = \Pi_{s}$. 
Then, $\mathcal{I}_{s} = (t_{0}, t_{1})$, where
\begin{align*}
t_{0} &\triangleq \inf \big\{t: \, t < s,~ - \Phi_{12}(t, s)^{-1} \Phi_{11}(t, s) \prec \Pi_{s}\big\}, \\
t_{1} &\triangleq \sup \big\{t: \, t > s,~ - \Phi_{12}(t, s)^{-1} \Phi_{11}(t, s) \succ \Pi_{s}\big\}.
\end{align*}
\end{corollary}


\subsection{Solution to the State Covariance Equation}

In this subsection we provide an explicit expression for the solution to the covariance matrix eqaution.
First, we need to specify an alternative expression for the state transition matrix $\Phi_{A-BB\t\Pi}(t, s)$ of $A(t) - B(t)B(t)\t\Pi(t)$.

\begin{lemma} \label{lem:phi-AB-pi-Q}
Let condition \eqref{cond:pi-Q-exist} hold so that $\Pi(t)$ exists on $[0, 1]$. 
The state transition matrix of $A(t)-B(t)B(t)\t\Pi(t)$ is given by
\begin{equation} \label{eqn:phi-AB-pi-Q}
\Phi_{A-BB\t\Pi}(t, s) = \Phi_{11}(t, s) + \Phi_{12}(t, s) \Pi(s), 
\end{equation}
for $s, t \in [0, 1]$.
\end{lemma}


\begin{proposition} \label{prp:sigma1-Q}
Let condition \eqref{cond:pi-Q-exist} hold so that $\Pi(t)$ exists on $[0, 1]$ and let  $\Sigma(0) = \Sigma_{0} \succeq 0$. 
Then, the solution $\Sigma(t)$ of \eqref{ode:sigma-Q-simp} for $t \in [0, 1]$ is,
\begin{multline} \label{sol:sigma-Q}
\Sigma(t) = \Phi_{A-BB\t\Pi}(t, 0) \Sigma_{0} \Phi_{A-BB\t\Pi}(t, 0)\t 
\\* 
+ \int_{0}^{t} \Phi_{A-BB\t\Pi}(t, s) C(s)D(s)C(s)\t \Phi_{A-BB\t\Pi}(t, s)\t \, \dd s,
\end{multline}
where $\Phi_{A-BB\t\Pi}(t, s)$ is given by \eqref{eqn:phi-AB-pi-Q} and $\Pi(t)$ is given by \eqref{sol:pi-Q}. 
In particular, for any $s \in [0, 1]$, as $\Pi(s) \to - \Phi_{12}(1, s)^{-1} \Phi_{11}(1, s)$, then $\Sigma(1) \to 0_{n \times n}$. 
Furthermore, if $\Sigma_{0} \succ 0$, then, as $\Pi(s) \to - \Phi_{12}(0, s)^{-1} \Phi_{11}(0, s)$, $\Sigma(1) \to +\infty$.
\end{proposition}

\begin{proof}
It is clear that \eqref{sol:sigma-Q} is the solution to \eqref{ode:sigma-Q-simp}. 
The continuity of $\Pi(t)$ in $\Pi(s)$ for any given $s \in [0, 1]$ implies that, as $\Pi(s) \to - \Phi_{12}(1, s)^{-1} \Phi_{11}(1, s)$, then $\Pi(t) \to - \Phi_{12}(1, t)^{-1} \Phi_{11}(1, t)$ for each $t \in [0, 1]$. 
Next, it follows from the dominated convergence theorem~\cite{bass2013real} that, 
as $\Pi(t) \to - \Phi_{12}(1, t)^{-1} \Phi_{11}(1, t)$ for each $t \in [0, 1]$, then $\Sigma(1) \to 0_{n \times n}$. 
Similarly, as $\Pi(s) \to - \Phi_{12}(0, s)^{-1} \Phi_{11}(0, s)$, then $\Pi(t) \to - \Phi_{12}(0, t)^{-1} \Phi_{11}(0, t)$ for each $t \in [0, 1]$. 
When $\Sigma_{0} \succ 0$, as $\Pi(0) \to - \Phi_{12}(0, 0^{+})^{-1} \Phi_{11}(0, 0) = - \infty$, then $\Sigma(1) \to +\infty$. 
This completes the proof. \hfill
\end{proof}


For the special case when $\Sigma_{0} \succ 0$ and $C(t)D(t)C(t)\t = B(t)B(t)\t$ for all $t \in [0, 1]$, we can compute that
\begin{align*}
&\int_{0}^{1} \Phi_{A-BB\t\Pi}(1, \tau) B(\tau)B(\tau)\t \Phi_{A-BB\t\Pi}(1, \tau)\t \, \dd\tau \\
&= - \Phi_{A-BB\t\Pi}(1, \tau) \Phi_{12}(1, \tau)\t \big|_{1}^{0} \\
&= - \Phi_{11}(1, 0) \Phi_{12}(1, 0)\t - \Phi_{12}(1, 0) \Pi(0) \Phi_{12}(1, 0)\t. 
\end{align*}
For notational simplicity, we temporarily let $\Pi(0)$ be denoted by $\Pi_{0}$, let $\Phi_{11}(1, 0)$ be denoted by $\Phi_{11}^{10}$, and let $\Phi_{12}(1, 0)$ be denoted by $\Phi_{12}^{10}$. 
Then,
\begin{align*}
&\big(\Phi_{12}^{10}\big)^{-1} \Sigma_{1} \big(\Phi_{12}^{10}\big)^{- \mbox{\tiny\sf T}} 
= - \Pi_{0} - \big(\Phi_{12}^{10}\big)^{-1} \Phi_{11}^{10} 
\\* 
&\hspace{12mm} + \Big(\big(\Phi_{12}^{10}\big)^{-1} \Phi_{11}^{10} + \Pi_{0}\Big) \Sigma_{0} \Big(\big(\Phi_{11}^{10}\big)\t \big(\Phi_{12}^{10}\big)^{- \mbox{\tiny\sf T}} + \Pi_{0}\Big) \\
&= \Sigma_{0}^{-\frac{1}{2}} \! \left(\left[\Sigma_{0}^{\frac{1}{2}} \Big(\Pi_{0} + \big(\Phi_{12}^{10}\big)^{-1} \Phi_{11}^{10} \Big) \Sigma_{0}^{\frac{1}{2}} - \frac{I_{n}}{2}\right]^{2} - \frac{I_{n}}{4}\right) \! \Sigma_{0}^{-\frac{1}{2}}.
\end{align*}
Since from Lemma \ref{lem:pi-Q-exist-sol}, $\Pi_{0} + \big(\Phi_{12}^{10}\big)^{-1} \Phi_{11}^{10} \prec 0$, we have 
\begin{equation*}
\Sigma_{0}^{\frac{1}{2}} \Big(\Pi_{0} + \big(\Phi_{12}^{10}\big)^{-1} \Phi_{11}^{10} \Big) \Sigma_{0}^{\frac{1}{2}} - \frac{I_{n}}{2} \prec 0.
\end{equation*}
It follows that 
\begin{multline*}
\Sigma_{0}^{\frac{1}{2}} \Big(\Pi_{0} + \big(\Phi_{12}^{10}\big)^{-1} \Phi_{11}^{10} \Big) \Sigma_{0}^{\frac{1}{2}} - \frac{I_{n}}{2} 
= 
\\* 
- \bigg(\Sigma_{0}^{\frac{1}{2}} \big(\Phi_{12}^{10}\big)^{-1} \Sigma_{1} \big(\Phi_{12}^{10}\big)^{- \mbox{\tiny\sf T}} \Sigma_{0}^{\frac{1}{2}} + \frac{I_{n}}{4}\bigg)^{\frac{1}{2}}.
\end{multline*}
Therefore, given $\Sigma(1) = \Sigma_{1} \succ 0$, $\Pi(0)$ is unique and is given by
\begin{align*}
\Pi(0) 
&= - \big(\Phi_{12}^{10}\big)^{-1} \Phi_{11}^{10} + \frac{\Sigma_{0}^{-1}}{2} 
\\* 
&\hspace{5mm} - \Sigma_{0}^{-\frac{1}{2}} \left(\frac{I_{n}}{4} + \Sigma_{0}^{\frac{1}{2}} \big(\Phi_{12}^{10}\big)^{-1} \Sigma_{1} \big(\Phi_{12}^{10}\big)^{- \mbox{\tiny\sf T}} \Sigma_{0}^{\frac{1}{2}} \right)^{\frac{1}{2}} \Sigma_{0}^{-\frac{1}{2}},
\end{align*}
which is the same solution reported in \cite{chen2018III}.


Proposition \ref{prp:sigma1-Q} gives an explicit map from $\Pi(0)$ to $\Sigma(1)$ for the coupled matrix ODEs \eqref{ode:pi-Q-simp} and \eqref{ode:sigma-Q-simp} with $\Sigma(0) = \Sigma_{0} \succ 0$. 
Specifically, define 
\begin{multline*}
f: \big\{\Pi_{0} \in \mathbb{R}^{n \times n}: \, \Pi_{0} = \Pi_{0}\t \prec - \Phi_{12}(1, 0)^{-1} \Phi_{11}(1, 0)\big\} 
\\* 
\to \big\{\Sigma_{1} \in \mathbb{R}^{n \times n}: \, \Sigma_{1} = \Sigma_{1}\t \succ 0\big\},
\end{multline*}
such that
\begin{align} \label{map:pi0-sigma1-Q}
&f(\Pi_{0}) = \Big(\Phi_{11}(1, 0) + \Phi_{12}(1, 0) \Pi_{0}\Big) 
\nonumber \\* 
&\hspace{10mm} \times \bigg[ \Sigma_{0} + \int_{0}^{1} \Big(\Phi_{11}(s, 0) + \Phi_{12}(s, 0) \Pi_{0}\Big)^{-1} 
C_{s}D_{s}C_{s}\t 
\nonumber \\* 
&\hspace{30mm} \times \Big(\Phi_{11}(s, 0)\t + \Pi_{0} \Phi_{12}(s, 0)\t \Big)^{-1} \dd s \bigg] 
\nonumber \\* 
&\hspace{30mm} \times \Big(\Phi_{11}(1, 0)\t + \Pi_{0} \Phi_{12}(1, 0)\t \Big),
\end{align}
where $C_{s} \triangleq C(s)$ and $D_{s} \triangleq D(s)$. 
Similarly to \cite{liu2022add}, when $n > 1$, the map $f$ may not be monotone in the Loewner order; however, when $n = 1$, $f$ is monotonically decreasing. 
In view of Proposition \ref{prp:sigma1-Q}, we obtain the following result.

\begin{theorem} \label{thm:suff-sigma1-Q}
If there exists an $n \times n$ symmetric matrix $\Pi_{0} \prec - \Phi_{12}(1, 0)^{-1} \Phi_{11}(1, 0)$ such that the desired terminal state covariance $\Sigma_{1}$ can be written as $\Sigma_{1} = f(\Pi_{0})$, then, the optimal control law for system \eqref{sde:mult-jump-simp} is 
\begin{equation*}
u^{*}(t) = - B(t)\t\Pi(t)x(t),
\end{equation*}
where $\Pi(t)$ is the unique solution to \eqref{ode:pi-Q-simp} with $\Pi(0) = \Pi_{0}$.
\end{theorem}


\subsection{Existence and Uniqueness of the Solution}

In this subsection we investigate the existence and uniqueness of the solution to the coupled matrix differential equations \eqref{bdr:sigma-Q}, \eqref{ode:pi-Q-simp}, \eqref{ode:sigma-Q-simp}.


To this end, let us compute the Jacobian of the map $f$ defined by \eqref{map:pi0-sigma1-Q}. 
For notational simplicity, let $\Phi_{11}(s, 0)$ be denoted by $\Phi_{11}^{s0}$, let $\Phi_{12}(s, 0)$ be denoted by $\Phi_{12}^{s0}$, and let the integrand in \eqref{map:pi0-sigma1-Q} be denoted by 
\begin{equation*}
P_{s} \triangleq 
\Big(\Phi_{11}^{s0} + \Phi_{12}^{s0} \Pi_{0}\Big)^{-1} 
C_{s}D_{s}C_{s}\t 
\Big(\big(\Phi_{11}^{s0}\big)\t + \Pi_{0} \big(\Phi_{12}^{s0}\big)\t \Big)^{-1} \! \succeq 0.
\end{equation*}
Let $\Delta \Pi_{0}$ denote a small increment in $\Pi_{0}$. 
Then, we can write \cite{henderson1981deriving}
\begin{align*}
&\Big(\Phi_{11}^{s0} + \Phi_{12}^{s0} \Pi_{0} + \Phi_{12}^{s0} \Delta \Pi_{0} \Big)^{-1} = \Big(\Phi_{11}^{s0} + \Phi_{12}^{s0} \Pi_{0} \Big)^{-1} 
\\* 
&\hspace{10mm} - \Big(\Phi_{11}^{s0} + \Phi_{12}^{s0} \Pi_{0} \Big)^{-1} \Phi_{12}^{s0} \Delta \Pi_{0} \Big(\Phi_{11}^{s0} + \Phi_{12}^{s0} \Pi_{0} \Big)^{-1} 
\\* 
&\hspace{10mm} + O\Big(\big\| \Delta \Pi_{0} \big\|^{2}\Big).
\end{align*}
After collecting all the first order terms of $\Delta \Pi_{0}$, we obtain
\begin{align*}
&f\big(\Pi_{0} + \Delta \Pi_{0}\big) - f\big(\Pi_{0}\big) = O\Big(\big\| \Delta \Pi_{0} \big\|^{2}\Big) 
\\* 
&\hspace{6mm} + \Phi_{12}^{10} \Delta \Pi_{0} \bigg[\Sigma_{0} + \int_{0}^{1} P_{s} \, \dd s\bigg] \Big(\big(\Phi_{11}^{10}\big)\t + \Pi_{0} \big(\Phi_{12}^{10}\big)\t \Big) 
\\* 
&\hspace{6mm} + \Big(\Phi_{11}^{10} + \Phi_{12}^{10} \Pi_{0} \Big) \bigg[\Sigma_{0} + \int_{0}^{1} P_{s} \, \dd s\bigg] \Delta \Pi_{0} \big(\Phi_{12}^{10}\big)\t 
\\* 
&\hspace{6mm} - \Big(\Phi_{11}^{10} + \Phi_{12}^{10} \Pi_{0} \Big) \bigg[\int_{0}^{1} \bigg( \Big(\Phi_{11}^{s0} + \Phi_{12}^{s0} \Pi_{0} \Big)^{-1} \Phi_{12}^{s0} \Delta \Pi_{0} P_{s} 
\\* 
&\hspace{18mm} + P_{s} \Delta \Pi_{0} \big(\Phi_{12}^{s0}\big)\t \Big(\big(\Phi_{11}^{s0}\big)\t + \Pi_{0} \big(\Phi_{12}^{s0}\big)\t  \Big)^{-1} \bigg) \, \dd s \bigg] 
\\* 
&\hspace{45mm} \times \Big(\big(\Phi_{11}^{10}\big)\t + \Pi_{0} \big(\Phi_{12}^{10}\big)\t \Big).
\end{align*}

For notational simplicity, let
\begin{equation*}
\Phi_{\Pi}^{10} \triangleq \Phi_{A-BB\t\Pi}(1, 0) = \Phi_{11}^{10} + \Phi_{12}^{10} \Pi_{0},
\end{equation*}
and
\begin{align*}
W_{s0} \triangleq 
\begin{cases}
\Big(\big(\Phi_{12}^{s0}\big)^{-1} \Phi_{11}^{s0} + \Pi_{0}\Big)^{-1} \prec 0, & s \in (0, 1]. 
\\* 
0_{n \times n}, & s = 0.
\end{cases}
\end{align*}
Then, we have
\begin{align} \label{eqn:map-diff-Q}
& f\big(\Pi_{0} + \Delta \Pi_{0}\big) - f\big(\Pi_{0}\big) = O\Big(\big\| \Delta \Pi_{0} \big\|^{2}\Big) 
\nonumber \\* 
& + \Phi_{\Pi}^{10} \bigg[ W_{10} \Delta \Pi_{0} \Sigma_{0} + \Sigma_{0} \Delta \Pi_{0} W_{10} 
\nonumber \\* 
& + \int_{0}^{1} \Big(W_{10} - W_{s0}\Big) \Delta \Pi_{0} P_{s} + P_{s} \Delta \Pi_{0} \Big(W_{10} - W_{s0}\Big) \, \dd s \bigg] 
\nonumber \\* 
&\hspace{62mm} \times \big(\Phi_{\Pi}^{10}\big)\t.
\end{align}


In the sequel, let $\otimes$ denote the Kronecker product. 
Given an $n \times n$ matrix $H = [h_{ij}]$, its vectorized version is
\begin{equation*}
\vect(H) \triangleq [h_{11} \enspace \dots \enspace h_{n1} \enspace h_{12} \enspace \dots \enspace h_{n2} \enspace \dots \enspace h_{1n} \enspace \dots \enspace h_{nn}]\t.
\end{equation*}
Define the map $\bar{f}: \big\{\vect(\Pi_{0}) \in \R^{n^{2}}: \, \Pi_{0} = \Pi_{0}\t \prec - \Phi_{12}(1, 0)^{-1} \Phi_{11}(1, 0)\big\} \to \big\{\vect(\Sigma_{1}) \in \R^{n^{2}}: \, \Sigma_{1} = \Sigma_{1}\t \succ 0\big\}$ such that 
\begin{equation*} \label{map:pi0-sigma1-Q-vec}
\bar{f}\big(\vect(\Pi_{0})\big) = \vect\big(f(\Pi_{0})\big),
\end{equation*}
where $f$ is defined by \eqref{map:pi0-sigma1-Q}.


It follows from vectorizing both sides of \eqref{eqn:map-diff-Q} that 
\begin{multline*}
\bar{f}\big(\vect(\Pi_{0}) + \vect(\Delta \Pi_{0})\big) - \bar{f}\big(\vect(\Pi_{0})\big) = 
\\* 
\partial \bar{f}\big(\vect(\Pi_{0})\big) \vect\big(\Delta \Pi_{0}\big) 
+ O\Big(\big\| \Delta \Pi_{0} \big\|^{2}\Big),
\end{multline*}
where,
\begin{multline} \label{eqn:map-jacob-Q}
\partial \bar{f}\big(\vect(\Pi_{0})\big) = \Phi_{\Pi}^{10} \otimes \Phi_{\Pi}^{10} \bigg[ \Sigma_{0} \otimes W_{10} + W_{10} \otimes \Sigma_{0} 
\\* 
+ \int_{0}^{1} P_{s} \otimes \big(W_{10} - W_{s0}\big) + \big(W_{10} - W_{s0}\big) \otimes P_{s} \, \dd s \bigg].
\end{multline}
Thus, $\partial \bar{f}\big(\vect(\Pi_{0})\big)$ is the Jacobian of the map $\bar{f}$ at $\vect(\Pi_{0})$.


\begin{lemma} \label{lem:map-inv-Q}
For any given $\Sigma_{0} \succ 0$, the map $f$ defined by \eqref{map:pi0-sigma1-Q} is a homeomorphism. 
Thus, for any $\Sigma_{1} \succ 0$, there exists a unique $\Pi_{0} \prec - \Phi_{12}(1, 0)^{-1} \Phi_{11}(1, 0)$ such that $\Sigma_{1} = f(\Pi_{0})$.
\end{lemma}


The proof of Lemma \ref{lem:map-inv-Q} is in the Appendix. 
An immediate result of Lemma \ref{lem:map-inv-Q} and Theorem \ref{thm:suff-sigma1-Q} is the following.

\begin{theorem} \label{thm:exist-unique-Q}
Let $\Sigma_{0}, \Sigma_{1} \succ 0$. 
The unique optimal control law that solves the covariance steering problem for system \eqref{sde:mult-jump-simp} is given by \eqref{ctrl:opt-Q}, where $\Pi(t)$ is the unique solution to \eqref{ode:pi-Q-simp}, \eqref{ode:sigma-Q-simp}, \eqref{bdr:sigma-Q}.
\end{theorem}


\section{Numerical Example} \label{sec:expl}

Generally speaking, we do not have a closed-form solution to the coupled matrix ODEs \eqref{ode:pi-Q}, \eqref{ode:sigma-Q}, \eqref{bdr:sigma-Q}, and thus of the optimal control law. 
Nevertheless, we can convert the optimal covariance steering problem into a semidefinite program and solve it numerically \cite[Section IV.A]{chen2016II}. 
When $E_{i}(t) \equiv I_{n}$ for all $i \in \{1, 2, \dots, \ell\}$, the coupled matrix ODEs \eqref{bdr:sigma-Q}, \eqref{ode:pi-Q-simp}, \eqref{ode:sigma-Q-simp} can be reduced to the case when $\nu(t) \equiv 0$ by letting $\bar{A}(t) \triangleq A(t) + \nu(t) I_{n}$. 
Thus, we can adopt numerical methods similar to those in \cite[Section VII.A]{liu2022add} to compute the optimal control.

We illustrate the results of the theory using the following example. 
Consider the linear stochastic system
\begin{equation*}
\dd x(t) = A x(t) \, \dd t + B u(t) \, \dd t + C \, \dd h(t) + x(t) \, \dd w(t),
\end{equation*}
where
\begin{equation*}
A = 
\begin{bmatrix}
-2 & 1 \\
0  & 0
\end{bmatrix}
, \quad
B = 
\begin{bmatrix}
0 \\
1
\end{bmatrix}
, \quad
C = 
\begin{bmatrix}
1 \\
0
\end{bmatrix},
\end{equation*}
$w(t)$ is a Wiener process, and $h(t)$ is a nonhomogeneous compound Poisson process with arrival rate $\lambda(t) = 3 + t$ and i.i.d. jump size $\chi \sim \mathcal{N}(0, 0.5^{2})$. 
Suppose the initial state is distributed according to $x(0) \sim \mathcal{N}\big(\left[0 \enspace 0\right]\t, \Sigma_{0}\big)$, where $\Sigma_{0} = \diag [1, 1]$. 
The objective is to reach a final state with zero mean and covariance $\Sigma_{1} = \diag [0.3, 0.2]$ over the interval $[0, 1]$, while minimizing the cost
\begin{equation*}
J(u) = \mathbb{E}\left[\int_{0}^{1} \Big(u^{2}(t) + x(t)\t \diag [1, 0] \, x(t)\Big) \, \dd t\right]. 
\end{equation*}
Figure \ref{fig:expl-3} illustrates ten controlled sample paths of the state $x(t) = \left[x_{1}(t) \enspace x_{2}(t)\right]\t$. 
The ``three standard deviation'' tolerance region on $[0, 1]$ is depicted as the transparent blue tube.
\begin{figure}[!h]
    \centering
    \includegraphics[width=0.48\textwidth]{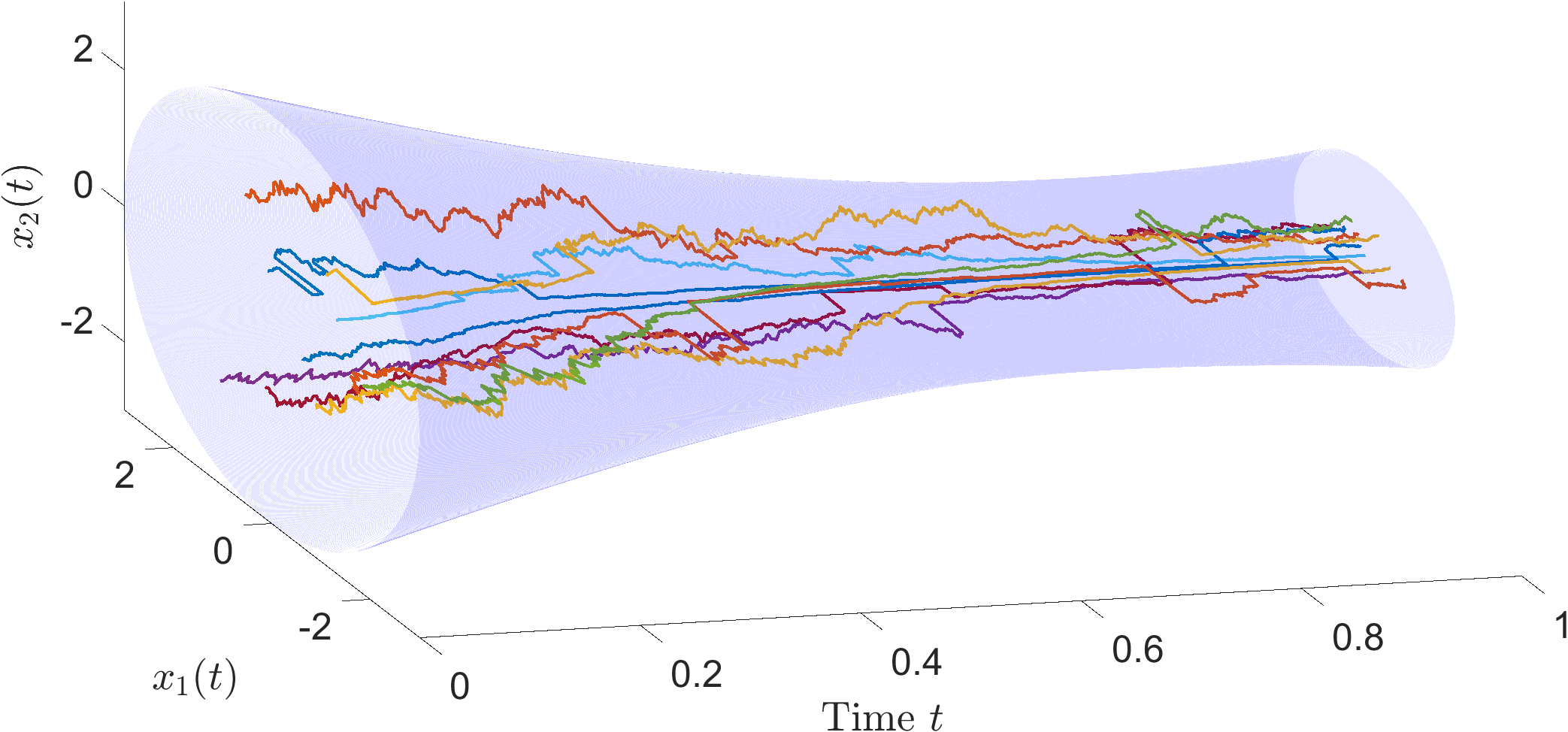}
    \caption{Sample paths.}
    \label{fig:expl-3}
\end{figure}


\section{Concluding Remarks}

In this paper, we derive the optimal control law for steering the state covariance of a linear time-varying stochastic system corrupted by both additive and multiplicative generic noise (possibly with jumps of any size). 
The cost functional is quadratic in both the state and control input. 
A necessary and sufficient condition for the existence of the solution to a matrix Riccati differential equation is established, provided that the linear time-varying system is totally controllable in the absence of noise. 
When the initial or the desired terminal state mean is nonzero, following the same procedure as in \cite{liu2022disc}, it is not difficult to show that 
the presence of multiplicative noise will result in the dependence of the optimal covariance steering on the optimal mean steering, while the presence of additive noise only will not. 
A potential future research direction would be to impose chance constraints along the sample paths and to optimize over the chance-constrained paths.


\bibliographystyle{IEEEtran}
\bibliography{TAC-conv-contr-mult-v8}

\begin{thebibliography}{10}
\providecommand{\url}[1]{#1}
\csname url@samestyle\endcsname
\providecommand{\newblock}{\relax}
\providecommand{\bibinfo}[2]{#2}
\providecommand{\BIBentrySTDinterwordspacing}{\spaceskip=0pt\relax}
\providecommand{\BIBentryALTinterwordstretchfactor}{4}
\providecommand{\BIBentryALTinterwordspacing}{\spaceskip=\fontdimen2\font plus
\BIBentryALTinterwordstretchfactor\fontdimen3\font minus
  \fontdimen4\font\relax}
\providecommand{\BIBforeignlanguage}[2]{{%
\expandafter\ifx\csname l@#1\endcsname\relax
\typeout{** WARNING: IEEEtran.bst: No hyphenation pattern has been}%
\typeout{** loaded for the language `#1'. Using the pattern for}%
\typeout{** the default language instead.}%
\else
\language=\csname l@#1\endcsname
\fi
#2}}
\providecommand{\BIBdecl}{\relax}
\BIBdecl

\bibitem{liu2022add}
F.~Liu and P.~Tsiotras, ``Optimal covariance steering for continuous-time
  linear stochastic systems with additive generic noise,''
  \emph{arXiv:2206.11201}, 2022.

\bibitem{chen2016I}
Y.~Chen, T.~T. Georgiou, and M.~Pavon, ``Optimal steering of a linear
  stochastic system to a final probability distribution, part {I},''
  \emph{{IEEE} Trans. Autom. Control}, vol.~61, no.~5, pp. 1158--1169, 2016.

\bibitem{chen2016II}
Y.~Chen, T.~T. Georgiou, and M.~Pavon, ``Optimal steering of a linear
  stochastic system to a final probability distribution, part {II},''
  \emph{{IEEE} Trans. Autom. Control}, vol.~61, no.~5, pp. 1170--1180, 2016.

\bibitem{chen2018III}
Y.~Chen, T.~T. Georgiou, and M.~Pavon, ``Optimal steering of a linear
  stochastic system to a final probability distribution, part {III},''
  \emph{{IEEE} Trans. Autom. Control}, vol.~63, no.~9, pp. 3112--3118, 2018.

\bibitem{volpe2016effective}
G.~Volpe and J.~Wehr, ``Effective drifts in dynamical systems with
  multiplicative noise: a review of recent progress,'' \emph{Rep. Prog. Phys.},
  vol.~79, no.~5, p. 053901, 2016.

\bibitem{harris1998signal}
C.~M. Harris and D.~M. Wolpert, ``Signal-dependent noise determines motor
  planning,'' \emph{Nature}, vol. 394, no. 6695, pp. 780--784, 1998.

\bibitem{bormetti2014multiplicative}
G.~Bormetti and S.~Cazzaniga, ``Multiplicative noise, fast convolution and
  pricing,'' \emph{Quant. Finance}, vol.~14, no.~3, pp. 481--494, 2014.

\bibitem{loo1968statistical}
S.~Loo, ``Statistical design of optimum multivariable systems whose inputs are
  corrupted by multiplicative noise,'' \emph{{IEEE} Trans. Autom. Control},
  vol.~13, no.~3, pp. 300--301, 1968.

\bibitem{mclane1971optimal}
P.~McLane, ``Optimal stochastic control of linear systems with state- and
  control-dependent disturbances,'' \emph{{IEEE} Trans. Autom. Control},
  vol.~16, no.~6, pp. 793--798, 1971.

\bibitem{wu2019preview}
J.~Wu, F.~Liao, and Z.~Xu, ``Preview control for a class of linear stochastic
  systems with multiplicative noise,'' \emph{Int. J. Syst. Sci.}, vol.~50,
  no.~14, pp. 2592--2603, 2019.

\bibitem{protter2003stochastic}
P.~E. Protter, \emph{Stochastic Integration and Differential Equations}.\hskip
  1em plus 0.5em minus 0.4em\relax Springer, 2003, vol.~21.

\bibitem{silverman1967controllability}
L.~M. Silverman and H.~Meadows, ``Controllability and observability in
  time-variable linear systems,'' \emph{{SIAM} J. Control}, vol.~5, no.~1, pp.
  64--73, 1967.

\bibitem{stubberud1964controllability}
A.~Stubberud, ``A controllability criterion for a class of linear systems,''
  \emph{{IEEE} Trans. Ind. Appl.}, vol.~83, no.~75, pp. 411--413, 1964.

\bibitem{morse1973structure}
A.~Morse and L.~Silverman, ``Structure of index-invariant systems,''
  \emph{{SIAM} J. Control}, vol.~11, no.~2, pp. 215--225, 1973.

\bibitem{bass2013real}
R.~F. Bass, \emph{Real Analysis for Graduate Students}.\hskip 1em plus 0.5em
  minus 0.4em\relax CreateSpace, 2013.

\bibitem{henderson1981deriving}
H.~V. Henderson and S.~R. Searle, ``On deriving the inverse of a sum of
  matrices,'' \emph{SIAM Rev.}, vol.~23, no.~1, pp. 53--60, 1981.

\bibitem{liu2022disc}
F.~Liu, G.~Rapakoulias, and P.~Tsiotras, ``Optimal covariance steering for
  discrete-time linear stochastic systems,'' \emph{arXiv:2211.00618}, 2022.

\bibitem{wonham1985linear}
W.~M. Wonham, \emph{Linear Multivariable Control: A Geometric Approach}.\hskip
  1em plus 0.5em minus 0.4em\relax New York, USA: Springer, 1985.

\bibitem{horn2012matrix}
R.~A. Horn and C.~R. Johnson, \emph{Matrix Analysis}.\hskip 1em plus 0.5em
  minus 0.4em\relax Cambridge University Press, 2012.

\bibitem{kilicaslan2010existence}
S.~Kilicaslan and S.~P. Banks, ``Existence of solutions of {R}iccati
  differential equations for linear time varying systems,'' in \emph{Proc.
  Amer. Control Conf.}, Baltimore, MD, 2010, pp. 1586--1590.

\bibitem{freiling1996generalized}
G.~Freiling, G.~Jank, and H.~Abou-Kandil, ``Generalized {R}iccati difference
  and differential equations,'' \emph{Linear Algebra Its Appl.}, vol. 241, pp.
  291--303, 1996.

\bibitem{perko1996differential}
L.~Perko, \emph{Differential Equations and Dynamical Systems}.\hskip 1em plus
  0.5em minus 0.4em\relax Springer, 1996.

\bibitem{kuvcera1973review}
V.~Ku{\v{c}}era, ``A review of the matrix {R}iccati equation,''
  \emph{Kybernetika}, vol.~9, no.~1, pp. 42--61, 1973.

\bibitem{krantz2002implicit}
S.~G. Krantz and H.~R. Parks, \emph{The Implicit Function Theorem: History,
  Theory, and Applications}.\hskip 1em plus 0.5em minus 0.4em\relax Springer,
  2002.

\end{thebibliography}


\section*{Appendix}

\subsection{Properties of the State Transition Matrix}

In Table \ref{tbl:compare}, we summarize a list of properties of the state transition matrix $\Phi_{M}$ (see \eqref{eqn:phi-M-blk}) for the case when $Q(t) \succeq 0$, and compare the properties with the case when $Q(t) \equiv 0$. 
The properties that require $\big(A(t), B(t)\big)$ to be totally controllable are shown in shaded color in the table. 
Note that $\mathcal{I}_{s} \subset \mathbb{R}$ is the maximal interval of existence of $\Pi(t)$, which is the solution to 
\begin{equation*}
\dot{\Pi} = - A(t)\t \Pi - \Pi A(t) + \Pi B(t) B(t)\t \Pi - Q(t),
\end{equation*}
starting from $\Pi(s) = \Pi_s$.
\begin{table*}[!t] 
\setlength\tabcolsep{6pt} 
\caption{Comparison between the cases when $Q(t) \succeq 0$ and $Q(t) \equiv 0$, assuming $\big(A(t), B(t)\big)$ is totally controllable.}
\label{tbl:compare}
\begin{center}
\begin{tabular}{ |c|c| }
\hline
Matrices With $Q(t) \succeq 0$ & Matrices With $Q(t) \equiv 0$ \\
\hline
$\Phi_{11}(t, s)$ & $\Phi_{A}(t, s)$ \\
$\Phi_{12}(t, s)$ & $- \Phi_{A}(t, s) N(t, s)$ \\
$\Phi_{21}(t, s)$ & $0_{n \times n}$ \\
$\Phi_{22}(t, s)$ & $\Phi_{A}(s, t)\t$ \\
$- \Phi_{11}(t, s)^{-1} \Phi_{12}(t, s)$ & $N(t, s) \triangleq \int_{s}^{t} \Phi_{A}(s, \tau) B(\tau) R(\tau)^{-1} B(\tau)\t \Phi_{A}(s, \tau)\t \, \dd \tau$ \\
$- \Phi_{12}(t, s) \Phi_{22}(t, s)^{-1}$ & $M(t, s) \triangleq \int_{s}^{t} \Phi_{A}(t, \tau) B(\tau) R(\tau)^{-1} B(\tau)\t \Phi_{A}(t, \tau)\t \, \dd \tau$ \\
\hline
\hline
Properties With $Q(t) \succeq 0$ & Properties With $Q(t) \equiv 0$ \\
\hline
$\Phi_{11}(t, s)$ and $\Phi_{22}(t, s)$ are invertible for $t, s \in \mathbb{R}$ & $\Phi_{A}(t, s)$ is invertible for $t, s \in \mathbb{R}$ \\
\rowcolor{myblue}
$\Phi_{12}(t, s)$ is invertible for $t \neq s$ & $- \Phi_{A}(t, s) N(t, s)$ is invertible for $t \neq s$ \\
$\Phi_{11}(t, s) = \Phi_{22}(s, t)\t$,~~ $t, s \in \mathbb{R}$ & $\Phi_{A}(t, s) = \Phi_{A}(t, s)$ \\
$\Phi_{12}(t, s) = - \Phi_{12}(s, t)\t$,~~ $t, s \in \mathbb{R}$ & $M(t, s) = - N(s, t)$,~~ $t, s \in \mathbb{R}$ \\
$\Phi_{21}(t, s) = - \Phi_{21}(s, t)\t$,~~ $t, s \in \mathbb{R}$ & $0_{n \times n} = 0_{n \times n}$ \\
$\Phi_{11}(t, s) \Phi_{12}(s, t) = - \Phi_{12}(t, s) \Phi_{22}(s, t)$,~~ $t, s \in \mathbb{R}$ & $M(t, s) = \Phi_{A}(t, s) N(t, s) \Phi_{A}(t, s)\t$,~~ $t, s \in \mathbb{R}$ \\
$\Phi_{21}(t, s) \Phi_{11}(s, t) = - \Phi_{22}(t, s) \Phi_{21}(s, t)$,~~ $t, s \in \mathbb{R}$ & $0_{n \times n} = 0_{n \times n}$ \\
$\Phi_{11}(t, s) \Phi_{11}(s, t) + \Phi_{12}(t, s) \Phi_{21}(s, t) = I_{n}$,~~ $t, s \in \mathbb{R}$ & $\Phi_{A}(t, s) \Phi_{A}(s, t) = I_{n}$,~~ $t, s \in \mathbb{R}$ \\
$\Phi_{21}(t, s) \Phi_{12}(s, t) + \Phi_{22}(t, s) \Phi_{22}(s, t) = I_{n}$,~~ $t, s \in \mathbb{R}$ & $\Phi_{A}(s, t)\t \Phi_{A}(t, s)\t = I_{n}$,~~ $t, s \in \mathbb{R}$ \\
$\Phi_{12}(t, s)\t \Phi_{22}(t, s) = \Phi_{22}(t, s)\t \Phi_{12}(t, s)$,~~ $t, s \in \mathbb{R}$ & $- N(t, s) = - N(t, s)$ \\
$\Phi_{21}(t, s)\t \Phi_{11}(t, s) = \Phi_{11}(t, s)\t \Phi_{21}(t, s)$,~~ $t, s \in \mathbb{R}$ & $0_{n \times n} = 0_{n \times n}$ \\
$\Phi_{12}(t, s) \Phi_{11}(t, s)\t = \Phi_{11}(t, s) \Phi_{12}(t, s)\t$,~~ $t, s \in \mathbb{R}$ & $- M(t, s) = - M(t, s)$ \\
$\Phi_{21}(t, s) \Phi_{22}(t, s)\t = \Phi_{22}(t, s) \Phi_{21}(t, s)\t$,~~ $t, s \in \mathbb{R}$ & $0_{n \times n} = 0_{n \times n}$ \\
$\Phi_{11}(t, s)\t \Phi_{22}(t, s) - \Phi_{21}(t, s)\t \Phi_{12}(t, s) = I_{n}$,~~ $t, s \in \mathbb{R}$ & $\Phi_{A}(t, s)\t \Phi_{A}(s, t)\t = I_{n}$,~~ $t, s \in \mathbb{R}$ \\
$\Phi_{11}(t, s) \Phi_{22}(t, s)\t - \Phi_{12}(t, s) \Phi_{21}(t, s)\t = I_{n}$,~~ $t, s \in \mathbb{R}$ & $\Phi_{A}(t, s) \Phi_{A}(s, t) = I_{n}$,~~ $t, s \in \mathbb{R}$ \\
\rowcolor{myblue}
$0 \prec - \Phi_{11}(t_{1}, s)^{-1} \Phi_{12}(t_{1}, s) \prec - \Phi_{11}(t_{2}, s)^{-1} \Phi_{12}(t_{2}, s)$ for $s < t_{1} < t_{2}$ & $0 \prec N(t_{1}, s) \prec N(t_{2}, s)$ for $s < t_{1} < t_{2}$ \\
\rowcolor{myblue}
$- \Phi_{11}(t_{1}, s)^{-1} \Phi_{12}(t_{1}, s) \prec - \Phi_{11}(t_{2}, s)^{-1} \Phi_{12}(t_{2}, s) \prec 0$ for $t_{1} < t_{2} < s$ & $N(t_{1}, s) \prec N(t_{2}, s) \prec 0$ for $t_{1} < t_{2} < s$ \\
\rowcolor{myblue}
$- \Phi_{12}(t, s_{1}) \Phi_{22}(t, s_{1})^{-1} \succ - \Phi_{12}(t, s_{2}) \Phi_{22}(t, s_{2})^{-1} \succ 0$ for $s_{1} < s_{2} < t$ & $M(t, s_{1}) \succ M(t, s_{2}) \succ 0$ for $s_{1} < s_{2} < t$ \\
\rowcolor{myblue}
$0 \succ - \Phi_{12}(t, s_{1}) \Phi_{22}(t, s_{1})^{-1} \succ - \Phi_{12}(t, s_{2}) \Phi_{22}(t, s_{2})^{-1}$ for $t < s_{1} < s_{2}$ & $0 \succ M(t, s_{1}) \succ M(t, s_{2})$ for $t < s_{1} < s_{2}$ \\
$\Phi_{A-BR^{-1}B\t\Pi}(t, s) = \Phi_{11}(t, s) + \Phi_{12}(t, s) \Pi(s)$,~~ $t \in \mathcal{I}_{s}$ & $\Phi_{A-BR^{-1}B\t\Pi}(t, s) = \Phi_{A}(t, s) - \Phi_{A}(t, s)N(t, s)\Pi(s)$,~~ $t \in \mathcal{I}_{s}$ \\
\hline
$\Phi_{12}(t, s) = - \Phi_{A-BR^{-1}B\t\Pi}(t, s) \bar{N}(t, s)$,~~ $t, s \in \mathbb{R}$, where & $\Phi_{A}(t, s) N(t, s) = \Phi_{A-BR^{-1}B\t\Pi}(t, s) \bar{N}(t, s)$,~~ $t, s \in \mathbb{R}$, where \\
\multicolumn{2}{|c|}{$\bar{N}(t, s) \triangleq \int_{s}^{t} \Phi_{A-BR^{-1}B\t\Pi}(s, \tau) B(\tau) R(\tau)^{-1} B(\tau)\t \Phi_{A-BR^{-1}B\t\Pi}(s, \tau)\t \, \dd \tau$} \\
\hline
\rowcolor{myblue}
$\mathcal{I}_{s} = (t_{0}, t_{1})$, where & $\mathcal{I}_{s} = (t_{0}, t_{1})$, where \\
\rowcolor{myblue}
$t_{0} = \inf \big\{t: \, t < s,~ - \Phi_{12}(t, s)^{-1} \Phi_{11}(t, s) \prec \Pi(s)\big\}$ & $t_{0} = \inf \big\{t: \, t < s,~ N(t, s)^{-1} \prec \Pi(s)\big\}$ \\
\rowcolor{myblue}
$t_{1} = \sup \big\{t: \, t > s,~ - \Phi_{12}(t, s)^{-1} \Phi_{11}(t, s) \succ \Pi(s)\big\}$ & $t_{1} = \sup \big\{t: \, t > s,~ N(t, s)^{-1} \succ \Pi(s)\big\}$ \\
\hline
\end{tabular}
\end{center}
\end{table*}


\subsection{Proof of Proposition \ref{prp:U}}

We need the following result to prove Proposition \ref{prp:U}.

\begin{lemma} \label{lem:u-bd}
Let $G$ and $H$ be given integers such that $G \geq H \geq 0$, 
let $\gamma, \alpha_{i}, \beta_{i} \in \mathbb{R}$, where $i \in \{0, 1, \dots, H\}$, 
and let $f \in \mathcal{C}^{G}$ be such that $f(t) > 0$ on $[0, 1]$. 
Let also $\rho \in \mathcal{C}^{0}$ be such that $\rho(0) < 0$ and $\rho(1) < \gamma$. 
There exists real-valued $u \in \mathcal{C}^{G}$ such that 
$\int_{0}^{1} f(\tau) u(\tau) \, \dd \tau = \gamma$ and 
$\int_{0}^{t} f(\tau)u(\tau) \, \dd \tau > \rho(t)$ for all $t \in [0,1]$, 
with boundary conditions $u(0) = \alpha_{0}$, $u(1) = \beta_{0}$, $\frac{\dd^{i} u}{\dd t^{i}}(0) = \alpha_{i}$, and $\frac{\dd^{i} u}{\dd t^{i}}(1) = \beta_{i}$, where $i \in \{1, \dots, H\}$. 
\end{lemma}

\begin{proof} 
We will construct a $\mathcal{C}^{G}$ function $u$ that satisfies all the conditions of the lemma. To this end, let 
\begin{multline*}
u(t) = \overbrace{\sum_{i=0}^{H} a_{i}t^{i}}^{a(t)} 
+ \overbrace{t^{H+1}\sum_{i=0}^{H} b_{i}(1-t)^{i}}^{b(t)} 
\\* 
+ \underbrace{c_{0}t^{H+1}(1-t)^{H+1}}_{c(t)} 
+ d(t),
\end{multline*}
where,
\begin{equation*}
d(t) \triangleq 
\begin{cases}
\frac{d_{0}(1-2t)}{t^{2} (1-t)^{2} f(t)} e^{-\frac{1}{t(1-t)}}, & t \in (0, 1). 
\\* 
0, & t \in \{0, 1\}.
\end{cases}
\end{equation*}
Let $\mathcal{H} \triangleq \{0, 1, \dots, H\}$. 
We show that the coefficients $a_{i}$, $b_{i}$, $c_{0}$, and $d_{0}$ are determined by $\alpha_{i}$, $\beta_{i}$, $\gamma$, and $\rho(t)$, respectively, where $i \in \mathcal{H}$.

First, notice that the derivatives of $d(t)$ up to order $G$ are zero at the boundaries $t=0$ and $t=1$, since the term $e^{-\frac{1}{t(1-t)}}$ dominates any polynomial terms at the boundaries and $f \in \mathcal{C}^{G}$. 
In addition, 
\begin{equation*}
\int_{0}^{t} f(\tau)d(\tau) \, \dd \tau = 
\begin{cases}
d_{0} e^{-\frac{1}{t(1-t)}}, & t \in (0, 1). 
\\* 
0, & t \in \{0, 1\}.
\end{cases}
\end{equation*}

To this end, we adopt the convention that the 0th derivative of a function is the function itself. 
Since $\frac{\dd^{i} u}{\dd t^{i}}(0) = \frac{\dd^{i} a}{\dd t^{i}}(0) = \alpha_{i}$ for $i \in \mathcal{H}$, the equations in the unknown coefficients $a_{i}$, where $i \in \mathcal{H}$, can be solved first. 
Since $\frac{\dd^{i} u}{\dd t^{i}}(1) = \frac{\dd^{i} (a+b)}{\dd t^{i}}(1) = \beta_{i}$ for $i \in \mathcal{H}$, the equations in the unknown coefficients $b_{i}$, where $i \in \mathcal{H}$, are in a triangular form. 
So $b_{i}$ can be solved next. 
Since $a(t)$ and $b(t)$ are known, $\int_{0}^{1} f(t) \big(a(t) + b(t)\big) \, \dd t$ is determined. 
It follows that $\int_{0}^{1} f(t)c(t) \, \dd t$ is fixed, from which $c_{0}$ can be solved.

Lastly, since $\rho(0) < 0$ and $\rho(1) < \gamma$, we can pick a suitable $d_{0} \geq 0$, so that $\int_{0}^{t} f(\tau)u(\tau) \, \dd \tau > \rho(t)$ for all $t \in (0,1)$. \hfill
\end{proof}


\begin{proof}[Proof of Proposition~\ref{prp:U}]
Let $A_{1} = 0$, $B_{1} = 1$, and 
\begin{equation} \label{eqn:AB-par}
A_{k+1} = 
\begin{bmatrix}
0_{k} & I_{k} \\
0 & 0_{k}\t
\end{bmatrix}
, ~
B_{k+1} = 
\begin{bmatrix}
0_{k} \\
1
\end{bmatrix}
, ~ 
k \in \mathbb{N},
\end{equation}
where $0_{k} \in \mathbb{R}^{k}$ is the zero column vector. 
Clearly, $A_{k+1}$ can be written as 
\begin{equation} \label{eqn:A-par}
A_{k+1} = 
\begin{bmatrix}
A_{k} & B_{k} \\
0_{k}\t & 0
\end{bmatrix}
, \quad 
k \in \mathbb{N}.
\end{equation}
Since $(A, B)$ is controllable, there exist a $p \times n$ matrix $\bar{F}$ and a vector $v \in \mathbb{R}^{p}$, such that $(A + B \bar{F}, B v)$ is controllable. 
It follows that there exists a coordinate transformation matrix $T$, such that $\big(T (A + B \bar{F}) T^{-1}, T B v\big)$ is in the control canonical form \cite{wonham1985linear}. 
It follows that there exists a row vector $g \in \mathbb{R}^{1 \times n}$, such that $\big(T (A + B \bar{F}) T^{-1} + T B v g, T B v\big) = (A_{n}, B_{n})$. 
Let $F \triangleq \bar{F} + v g T$. 
Then $\big(T (A + B F) T^{-1}, T B v\big) = (A_{n}, B_{n})$. 
In view of \eqref{ode:sigma-ABK}, it is without loss of generality to assume that $(A, B)$ is in the canonical form $(A_{n}, B_{n})$. 
We will prove the proposition by induction on the dimension $n$.

When $n=1$, $A_{1} = 0$ and $B_{1} = 1$, and thus
\begin{equation} \label{ode:Sig-1d}
\dot{\Sigma} = 2U(t) + M(t) + 2\nu(t) \Sigma.
\end{equation}
Hence,
\begin{equation*}
\Sigma(1) = e^{2\int_{0}^{1} \nu(t) \, \dd t} \Sigma(0) + \int_{0}^{1} e^{2\int_{t}^{1} \nu(\tau) \, \dd \tau} \big( 2U(t) + M(t) \big) \, \dd t.
\end{equation*}
Given $\Sigma(0) = \Sigma_{0}$, $\Sigma(1) = \Sigma_{1}$, and $\mathcal{C}^{H}$ functions $M(t) \succeq 0$ and $\nu(t) \geq 0$, it follows that $\int_{0}^{1} e^{2\int_{t}^{1} \nu(\tau) \, \dd \tau} U(t) \, \dd t$ is fixed. 
In view of Lemma \ref{lem:u-bd}, there exists a $\mathcal{C}^{H}$ function $U(t)$ on $[0, 1]$, which satisfies the preassigned boundary conditions $\frac{\dd^{i} U}{\dd t^{i}}(0)$ and $\frac{\dd^{i} U}{\dd t^{i}}(1)$, where $i \in \mathcal{H}$, such that $\Sigma(0) = \Sigma_{0}$, $\Sigma(1) = \Sigma_{1}$, and $\Sigma(t) > 0$ for all $t \in [0, 1]$. 
Hence, Proposition \ref{prp:U} holds for $n=1$.

Assume Proposition \ref{prp:U} holds for $n=k \in \mathbb{N}$. 
We now show that Proposition \ref{prp:U} also holds for $n=k+1$. 
To this end, let $M(t), \nu(t) \in \mathcal{C}^{H + k}$. 
We first partition $\Sigma(t)$, $M(t)$, and $U(t)$ as follows. 
\begin{equation} \label{eqn:Sig-M-U-par}
\Sigma(t) = 
\begin{bmatrix}
\Sigma_{\Box} & \Sigma_{\dagger} \\
\Sigma_{\dagger}\t & \Sigma_{\star}
\end{bmatrix}
, ~
M(t) = 
\begin{bmatrix}
M_{\Box} & M_{\dagger} \\
M_{\dagger}\t & M_{\star}
\end{bmatrix}
, ~
U(t) = 
\begin{bmatrix}
U_{\dagger} \\ 
U_{\star}
\end{bmatrix},
\end{equation}
where $\Sigma_{\Box}, M_{\Box} \in \mathbb{R}^{k \times k}$, $\Sigma_{\dagger}, M_{\dagger}, U_{\dagger} \in \mathbb{R}^{k}$, and $\Sigma_{\star}$, $M_{\star}$, $U_{\star} \in \mathbb{R}$. 
Then, by \eqref{ode:sigma-ABU}, \eqref{eqn:AB-par}, \eqref{eqn:A-par}, \eqref{eqn:Sig-M-U-par}, we can write
\begin{align}
\dot{\Sigma}_{\Box} 
&= A_{k} \Sigma_{\Box} + \Sigma_{\Box} A_{k}\t + B_{k} \Sigma_{\dagger}\t + \Sigma_{\dagger} B_{k}\t + M_{\Box}(t) 
\nonumber \\* 
&\hspace{53mm} + 2\nu(t) \Sigma_{\Box}, 
\label{ode:Sig-ul} \\
\dot{\Sigma}_{\dagger}
&= A_{k} \Sigma_{\dagger} + B_{k} \Sigma_{\star}(t) + U_{\dagger}(t) + M_{\dagger}(t) + 2\nu(t) \Sigma_{\dagger}, 
\label{ode:Sig-ur} \\
\dot{\Sigma}_{\star}
&= 2 U_{\star}(t) + M_{\star}(t) + 2\nu(t) \Sigma_{\star}. 
\label{ode:Sig-lr}
\end{align}
Notice that \eqref{ode:Sig-ul} has the same form as \eqref{ode:sigma-ABU}, except that $\Sigma_{\dagger}$ serves as the control input for $\Sigma_{\Box}$. 
Moreover, \eqref{ode:Sig-lr} has the same form as \eqref{ode:Sig-1d}. 
In view of \eqref{ode:Sig-ur} and \eqref{ode:Sig-lr}, the control input $U$ directly controls $\Sigma_{\dagger}$ and $\Sigma_{\star}$.

Notice that $\Sigma(0) = \Sigma_{0}$ and $\Sigma(1) = \Sigma_{1}$ are given, 
that $U(t)$'s boundary conditions $\frac{\dd^{i} U}{\dd t^{i}}(0)$ and $\frac{\dd^{i} U}{\dd t^{i}}(1)$ are given for $i \in \mathcal{H}$, 
and that $M(t)$'s and $\nu(t)$'s boundary conditions $\frac{\dd^{j} M}{\dd t^{j}}(0)$, $\frac{\dd^{j} M}{\dd t^{j}}(1)$, $\frac{\dd^{j} \nu}{\dd t^{i}}(0)$, and $\frac{\dd^{j} \nu}{\dd t^{j}}(1)$ can be calculated for $j \in \mathcal{H}$. 
Then, it follows from \eqref{ode:Sig-ur} and \eqref{ode:Sig-lr}, by induction, that $\Sigma_{\dagger}(t)$'s and $\Sigma_{\star}(t)$'s boundary conditions $\frac{\dd^{i} \Sigma_{\dagger}}{\dd t^{i}}(0)$, $\frac{\dd^{i} \Sigma_{\dagger}}{\dd t^{i}}(1)$, $\frac{\dd^{i} \Sigma_{\star}}{\dd t^{i}}(0)$, and $\frac{\dd^{i} \Sigma_{\star}}{\dd t^{i}}(1)$ are fixed for $i \in \{0, 1, \dots, H+1\}$. 
Since $\Sigma_{\Box}$ is of size $k \times k$ and, in view of \eqref{ode:Sig-ul}, $\Sigma_{\dagger}$ plays the role of the control input for $\Sigma_{\Box}$, the induction assumption applies. 
Hence, Proposition \ref{prp:U} holds for $H + 1$. 
Thus, there exists a $\mathcal{C}^{H + 1}$ function $\Sigma_{\dagger}(t)$ on $[0, 1]$, which satisfies the fixed boundary conditions $\frac{\dd^{i} \Sigma_{\dagger}}{\dd t^{i}}(0)$ and $\frac{\dd^{i} \Sigma_{\dagger}}{\dd t^{i}}(1)$, where $i \in \{0, 1, \dots, H+1\}$, such that $\Sigma_{\Box}(t)$ satisfies $\Sigma_{\Box}(t) \succ 0$ for all $t \in [0, 1]$ and the boundary conditions $\Sigma_{\Box}(0)$ and $\Sigma_{\Box}(1)$. 
Since \eqref{ode:Sig-lr} has the same form as \eqref{ode:Sig-1d}, by Lemma \ref{lem:u-bd}, there exists a $\mathcal{C}^{H}$ function $U_{\star}(t)$ on $[0, 1]$, which satisfies the preassigned boundary conditions $\frac{\dd^{i} U_{\star}}{\dd t^{i}}(0)$ and $\frac{\dd^{i} U_{\star}}{\dd t^{i}}(1)$, where $i \in \mathcal{H}$, such that $\Sigma_{\star}(t)$ satisfies $\Sigma_{\star}(t) > \Sigma_{\dagger}(t)\t \Sigma_{\Box}^{-1}(t) \Sigma_{\dagger}(t)$ for all $t \in [0, 1]$ and the boundary conditions $\Sigma_{\star}(0)$ and $\Sigma_{\star}(1)$. 
By Sylvester's criterion \cite{horn2012matrix}, $\Sigma(t) \succ 0$ for all $t \in [0, 1]$. 
With $\Sigma_{\dagger}(t), \Sigma_{\star}(t) \in \mathcal{C}^{H + 1}$, we can determine $U_{\dagger}(t) \in \mathcal{C}^{H}$ from \eqref{ode:Sig-ur}, whose boundary conditions $\frac{\dd^{i} U_{\dagger}}{\dd t^{i}}(0)$ and $\frac{\dd^{i} U_{\dagger}}{\dd t^{i}}(1)$, where $i \in \mathcal{H}$, are automatically satisfied from our previous construction when $\frac{\dd^{j} \Sigma_{\dagger}}{\dd t^{j}}(0)$ and $\frac{\dd^{j} \Sigma_{\dagger}}{\dd t^{j}}(1)$ are satisfied for $j \in \{0, 1, \dots, H+1\}$. 
Therefore, Proposition \ref{prp:U} also holds for $n=k+1$. \hfill
\end{proof}


\subsection{Proofs of Lemmas}

First, we introduce a result from \cite{kilicaslan2010existence} on the existence of solution to the matrix Riccati differential equation \eqref{ode:pi-Q-simp}.
\begin{proposition} \label{prp:pi-Q-exist}
Given an initial condition $\Pi(0) = \Pi_{0}$, \eqref{ode:pi-Q-simp} has a solution on the interval $[0, 1]$ if and only if $\Phi_{11}(t, 0) + \Phi_{12}(t, 0) \Pi_{0}$ is invertible for all $t \in [0, 1]$. 
The solution $\Pi(t)$ is unique on $[0, 1]$ and is given by
\begin{equation*}
\Pi(t) = \! \Big(\Phi_{21}(t, 0) + \Phi_{22}(t, 0) \Pi_{0}\Big) \! \Big(\Phi_{11}(t, 0) + \Phi_{12}(t, 0) \Pi_{0}\Big)^{-1}.
\end{equation*}
\end{proposition}

\begin{proof}[Proof of Lemma \ref{lem:phi-inv}]
First, we show that $\Phi_{11}(t, s)$ is invertible. 
Let $\Pi(t)$ be the solution to \eqref{ode:pi-Q-simp} with $\Pi(0) = 0_{n \times n}$, and
let $\mathcal{I}_{0}$ be the maximal interval of existence of $\Pi(t)$. 
Let $\Pi_{1}(t)$ be the solution to \eqref{ode:pi-Q-simp} when $Q(t) \equiv 0_{n \times n}$ and $\Pi(0) = 0_{n \times n}$. 
Clearly, we have $\Pi_{1}(t) \equiv 0_{n \times n}$ for all $t \in \mathbb{R}$. 
Let $\Pi_{2}(t)$ be the solution to \eqref{ode:pi-Q-simp} when $B(t) \equiv 0_{n \times n}$ and $\Pi(0) = 0_{n \times n}$. 
Since \eqref{ode:pi-Q-simp} is linear in this case, and it follows that the solution is given by
\begin{equation*}
\Pi_{2}(t) = - \int_{0}^{t} \Phi_{A}(\tau, t)\t Q(\tau) \Phi_{A}(\tau, t) \, \dd \tau, \quad t \in \mathbb{R}.
\end{equation*}
By the monotonicity of the matrix Riccati differential equation \cite{freiling1996generalized}, we have
\begin{align*}
&\Pi_{1}(t) \preceq \Pi(t) \preceq \Pi_{2}(t), \quad t \in \mathcal{I}_{0} \cap (- \infty, 0], \\
&\Pi_{2}(t) \preceq \Pi(t) \preceq \Pi_{1}(t), \quad t \in \mathcal{I}_{0} \cap [0, + \infty).
\end{align*}
Hence, $\Pi(t)$ has no finite escape time and $\mathcal{I}_{0} = \mathbb{R}$. 
Indeed, suppose $\mathcal{I}_{0} = (t_{0}, t_{1})$, where $t_{0} < 0 < t_{1}$. 
For $t \in [0, t_{1})$, and since $\Pi_{2}(t_{1}) \preceq \Pi_{2}(t) \preceq 0$, it follows that $\Pi_{2}(t_{1}) \preceq \Pi(t) \preceq 0$. 
Clearly, the set $\{\Pi \in \mathbb{R}^{n \times n}: \, \Pi_{2}(t_{1}) \preceq \Pi = \Pi\t \preceq 0\}$ is compact in $\mathbb{R}^{n \times n}$. 
Hence, we have $t_{1} = + \infty$ \cite{perko1996differential}. 
It can be similarly shown that $t_{0} = - \infty$. 
Thus, $\mathcal{I}_{0} = \mathbb{R}$. 
By Proposition \ref{prp:pi-Q-exist}, $\Phi_{11}(t, 0)$ is invertible for all $t \in \mathbb{R}$. A similar argument shows that $\Phi_{11}(t, s)$ is invertible for all $t, s \in \mathbb{R}$.

Next, we show that $\Phi_{11}(t, s) = \Phi_{22}(s, t)\t$. By \eqref{eqn:phi-blk-11}, it follows that
\begin{align*}
\Phi_{11}(t, s)^{-1} 
&= \Phi_{22}(t, s)\t - \Phi_{11}(t, s)^{-1} \Phi_{12}(t, s) \Phi_{21}(t, s)\t \\
&= \Phi_{22}(t, s)\t - \Phi_{12}(t, s)\t \Phi_{11}(t, s)^{- \mbox{\tiny\sf T}} \Phi_{21}(t, s)\t. 
\end{align*}
Since the matrix $\Phi_{11}(t, s)$ is invertible, the Schur complement of $\Phi_{11}(t, s)$ in $\Phi_{M}(t, s)$ is given by 
\begin{equation*}
\Phi_{22}(t, s) - \Phi_{21}(t, s) \Phi_{11}(t, s)^{-1} \Phi_{12}(t, s) = \Phi_{22}(s, t)^{-1}. 
\end{equation*}
Therefore, we have $\Phi_{11}(t, s) = \Phi_{22}(s, t)\t$. 
The other equality can be shown similarly. \hfill
\end{proof}


\begin{proof}[Proof of Lemma \ref{lem:phi-12-21}] 
By equations \eqref{eqn:phi-blk-12} and \eqref{eqn:phi-11-22}, 
\begin{align*}
\Phi_{12}(t, s) 
&= \Phi_{11}(t, s) \Phi_{12}(t, s)\t \Phi_{11}(t, s)^{- \mbox{\tiny\sf T}} 
\\ 
&= \Phi_{22}(s, t)\t \Phi_{12}(t, s)\t \Phi_{11}(t, s)^{- \mbox{\tiny\sf T}}. 
\end{align*}
Since $\Phi_{M}(t, s) \Phi_{M}(s, t) = I_{2n}$, we have $\Phi_{11}(t, s) \Phi_{12}(s, t) + \Phi_{12}(t, s) \Phi_{22}(s, t) = 0$. 
Thus, 
\begin{equation*}
- \Phi_{12}(s, t)\t 
= \Phi_{22}(s, t)\t \Phi_{12}(t, s)\t \Phi_{11}(t, s)^{- \mbox{\tiny\sf T}} 
= \Phi_{12}(t, s). 
\end{equation*}
Similarly, it follows from equations \eqref{eqn:phi-blk-21} and \eqref{eqn:phi-11-22} that 
\begin{align*}
\Phi_{21}(t, s)
&= \Phi_{22}(t, s) \Phi_{21}(t, s)\t \Phi_{22}(t, s)^{- \mbox{\tiny\sf T}} 
\\ 
&= \Phi_{11}(s, t)\t \Phi_{21}(t, s)\t \Phi_{22}(t, s)^{- \mbox{\tiny\sf T}}. 
\end{align*}
Since $\Phi_{21}(t, s) \Phi_{11}(s, t) + \Phi_{22}(t, s) \Phi_{21}(s, t) = 0$, we obtain 
\begin{equation*}
- \Phi_{21}(s, t)\t 
= \Phi_{11}(s, t)\t \Phi_{21}(t, s)\t \Phi_{22}(t, s)^{- \mbox{\tiny\sf T}} 
= \Phi_{21}(t, s). 
\end{equation*}
This completes the proof. \hfill
\end{proof}


\begin{proof}[Proof of Lemma \ref{lem:phi-inv-mono}]
First, since 
\begin{equation*}
\Phi_{11}(t, s) \Phi_{12}(s, t) + \Phi_{12}(t, s) \Phi_{22}(s, t) = 0,
\end{equation*}
we have \eqref{eqn:phi-12}.  
Since $\Phi_{11}(t, s)$ is invertible for all $t, s \in \mathbb{R}$, we show that $\Phi_{12}(t, s)$ is invertible for $t \neq s$ by showing that $- \Phi_{11}(t, s)^{-1} \Phi_{12}(t, s)$ is invertible for $t \neq s$. 
We can compute that
\begin{align*}
&\frac{\partial}{\partial t} \Big(- \Phi_{11}(t, s)^{-1} \Phi_{12}(t, s)\Big) 
= - \Phi_{11}(t, s)^{-1} \Big(\frac{\partial}{\partial t} \Phi_{12}(t, s)\Big) 
\\* 
&\hspace{15mm} + \Phi_{11}(t, s)^{-1} \Big(\frac{\partial}{\partial t} \Phi_{11}(t, s)\Big) \Phi_{11}(t, s)^{-1} \Phi_{12}(t, s) \\
&= \Phi_{11}(t, s)^{-1} B(t) B(t)\t \Phi_{11}(t, s)^{- \mbox{\tiny\sf T}}.
\end{align*}
Hence,
\begin{multline} \label{eqn:phi-11-12}
- \Phi_{11}(t, s)^{-1} \Phi_{12}(t, s) = 
\\* 
\int_{s}^{t} \Phi_{11}(\tau, s)^{-1} B(\tau) B(\tau)\t \Phi_{11}(\tau, s)^{- \mbox{\tiny\sf T}} \, \dd \tau.
\end{multline}
Assume that there exist $t, s \in \mathbb{R}$ with $s < t$, such that $- \Phi_{11}(t, s)^{-1} \Phi_{12}(t, s) \succeq 0$ is singular. 
Then, there exists $x \in \mathbb{R}^{n}$ with $x \neq 0$, such that $x\t \Big(- \Phi_{11}(t, s)^{-1} \Phi_{12}(t, s)\Big) x = 0$. 
It follows from \eqref{eqn:phi-11-12} that
\begin{align*}
& x\t \Phi_{11}(\tau, s)^{-1} B(\tau) \equiv 0, \quad \tau \in (s, t), \\
& \frac{\partial^{k}}{\partial \tau^{k}} \Big(x\t \Phi_{11}(\tau, s)^{-1} B(\tau)\Big) \equiv 0, ~~ \tau \in (s, t), ~~ 1 \leq k \leq n-1.
\end{align*}
It follows that
\begin{equation} \label{eqn:omega-n}
x\t \Phi_{11}(\tau, s)^{-1} \Omega_{n}(\tau) \equiv 0, \quad \tau \in (s, t),
\end{equation}
where,
\begin{align*}
\Omega_{n}(t) &\triangleq 
\begin{bmatrix}
\Upsilon_{0}(t) & \Upsilon_{1}(t) & \cdots & \Upsilon_{n-1}(t)
\end{bmatrix}
, \\
\Upsilon_{0}(t) &\triangleq B(t)
, \\ 
\Upsilon_{k}(t) &\triangleq \! \Big(- A(t) + B(t) B(t)\t \Phi_{21}(t, s) \Phi_{11}(t, s)^{-1} \Big) \Upsilon_{k-1}(t) 
\\* 
&\hspace{36mm} + \dot{\Upsilon}_{k-1}(t)
, ~~~ 1 \leq k \leq n-1.
\end{align*}
Clearly, $\range \Omega_{n}(t) = \range \Theta_{n}(t)$ for all $t \in \mathbb{R}$. 
Thus, \eqref{eqn:omega-n} implies that 
\begin{equation} \label{eqn:theta-n}
x\t \Phi_{11}(\tau, s)^{-1} \Theta_{n}(\tau) \equiv 0, \quad \tau \in (s, t).
\end{equation}
Since $\big(A(t), B(t)\big)$ is totally controllable on $[0, 1]$, there exists $t_{*} \in (s, t)$ such that $\rank \Theta_{n}(t_{*}) = n$. 
Since $\Phi_{11}(t_{*}, s)^{-1}$ is invertible, it follows from \eqref{eqn:theta-n} that $x = 0$, which contradicts the assumption that $x \neq 0$. 
Thus, $- \Phi_{11}(t, s)^{-1} \Phi_{12}(t, s) \succ 0$ for all $s < t$. 
A similar argument shows that $- \Phi_{11}(t, s)^{-1} \Phi_{12}(t, s) \prec 0$ for all $t < s$. 
Therefore, $- \Phi_{11}(t, s)^{-1} \Phi_{12}(t, s)$ is invertible for all $t \neq s$, and $\Phi_{12}(t, s)$ is invertible for all $t \neq s$. 
Furthermore, in light of \eqref{eqn:phi-11-12}, we can use a similar argument to show \eqref{ineq:phi-mono-1} and \eqref{ineq:phi-mono-2}. \hfill
\end{proof}


\begin{proof}[Proof of Lemma \ref{lem:pi-Q-exist-sol}]
First, we show \eqref{ineq:pi-Q-bd}. 
Notice that, for any fixed $r$, $- \Phi_{12}(r, t)^{-1} \Phi_{11}(r, t) = \Phi_{22}(t, r) \Phi_{12}(t, r)^{-1}$ satisfies the same Riccati differential equation as $\Pi(t)$, that is,
\begin{align*}
&\frac{\partial}{\partial t} \Big(\Phi_{22}(t, r) \Phi_{12}(t, r)^{-1}\Big) 
= - A(t)\t \Big(\Phi_{22}(t, r) \Phi_{12}(t, r)^{-1}\Big) 
\\* 
&\hspace{20mm} - \Big(\Phi_{22}(t, r) \Phi_{12}(t, r)^{-1}\Big) A(t) - Q(t) 
\\* 
&\hspace{1mm} + \Big(\Phi_{22}(t, r) \Phi_{12}(t, r)^{-1}\Big) B(t) B(t)\t \Big(\Phi_{22}(t, r) \Phi_{12}(t, r)^{-1}\Big).
\end{align*}
In light of the monotonicity of the matrix Riccati differential equation \cite{freiling1996generalized}, it follows that $- \Phi_{12}(0, s)^{-1} \Phi_{11}(0, s) \prec \Pi(s) \prec - \Phi_{12}(1, s)^{-1} \Phi_{11}(1, s)$ for some $s \in [0, 1]$ implies that, for all $t \in [0, 1]$,
$- \Phi_{12}(0, t)^{-1} \Phi_{11}(0, t) \prec \Pi(t) \prec - \Phi_{12}(1, t)^{-1} \Phi_{11}(1, t)$.
Hence, \eqref{ineq:pi-Q-bd} holds, which implies that \eqref{cond:pi-Q-exist} holds if and only if 
\begin{equation} \label{ineq:pi-Q-0}
\Pi(0) \prec - \Phi_{12}(1, 0)^{-1} \Phi_{11}(1, 0).
\end{equation}

In view of Proposition \ref{prp:pi-Q-exist}, it suffices to show that $\Phi_{11}(t, 0) + \Phi_{12}(t, 0) \Pi(0)$ is invertible for all $t \in [0, 1]$ if and only if \eqref{ineq:pi-Q-0} holds. 
By \eqref{ineq:phi-mono-1}, for $0 < t_{1} < t_{2} \leq 1$,
\begin{equation*}
\Phi_{12}(t_{1}, 0)^{-1} \Phi_{11}(t_{1}, 0) \prec \Phi_{12}(t_{2}, 0)^{-1} \Phi_{11}(t_{2}, 0) \prec 0.
\end{equation*}
Thus, the matrix inequality \eqref{ineq:pi-Q-0} implies that, for all $t \in (0, 1]$,
$\Phi_{12}(t, 0)^{-1} \Phi_{11}(t, 0) + \Pi(0) \prec 0$. 
Since $\Phi_{11}(0, 0) + \Phi_{12}(0, 0) \Pi(0) = I_{n}$ is invertible and 
\begin{multline} \label{eqn:phi-12-pi-0}
\Phi_{11}(t, 0) + \Phi_{12}(t, 0) \Pi(0) = 
\\* 
\Phi_{12}(t, 0) \Big(\Phi_{12}(t, 0)^{-1} \Phi_{11}(t, 0) + \Pi(0)\Big), ~~ t \in (0, 1],
\end{multline}
is invertible, it follows that $\Phi_{11}(t, 0) + \Phi_{12}(t, 0) \Pi(0)$ is invertible for all $t \in [0, 1]$. 
This shows the sufficiency of \eqref{ineq:pi-Q-0}. 
Next, we show the necessity of \eqref{ineq:pi-Q-0}. 
By \eqref{eqn:phi-12-pi-0}, $\Phi_{12}(t, 0)^{-1} \Phi_{11}(t, 0) + \Pi(0)$ is invertible for all $t \in (0, 1]$. 
Since 
\begin{equation*}
\lim_{t \to 0^{+}} \Phi_{12}(t, 0)^{-1} \Phi_{11}(t, 0) + \Pi(0) = - \infty,
\end{equation*}
it must be true that $\Phi_{12}(t, 0)^{-1} \Phi_{11}(t, 0) + \Pi(0) \prec 0$ for all $t \in (0, 1]$. 
Thus, \eqref{ineq:pi-Q-0} holds. 
Therefore, \eqref{cond:pi-Q-exist} is necessary and sufficient for \eqref{ode:pi-Q-simp} to admit a unique solution $\Pi(t)$ on $[0, 1]$.

Lastly, we show \eqref{sol:pi-Q}. 
It follows from Proposition \ref{prp:pi-Q-exist} that \eqref{sol:pi-Q} holds for $t \in [s, 1]$. 
Likewise, \eqref{sol:pi-Q} holds for $t \in [0, s]$ \cite{kuvcera1973review}. 
Thus, the unique solution $\Pi(t)$ is given by \eqref{sol:pi-Q}. \hfill
\end{proof}


\begin{proof}[Proof of Lemma \ref{lem:phi-AB-pi-Q}]
It is straightforward to verify that
\begin{equation*}
\Phi_{11}(s, s) + \Phi_{12}(s, s) \Pi(s) = I_{n}.
\end{equation*}
Since,
\begin{align*}
\frac{\partial}{\partial t} \Phi_{11}(t, s) &= A(t) \Phi_{11}(t, s) - B(t)B(t)\t \Phi_{21}(t, s), \\
\frac{\partial}{\partial t} \Phi_{12}(t, s) &= A(t) \Phi_{12}(t, s) - B(t)B(t)\t \Phi_{22}(t, s),
\end{align*}
in light of \eqref{sol:pi-Q}, we can compute that
\begin{align*}
&\frac{\partial}{\partial t} \Big(\Phi_{11}(t, s) + \Phi_{12}(t, s) \Pi(s)\Big) \\
&= A(t)\Big(\Phi_{11}(t, s) + \Phi_{12}(t, s) \Pi(s)\Big) 
\\* 
&\qquad\qquad - B(t)B(t)\t\Big(\Phi_{21}(t, s) + \Phi_{22}(t, s) \Pi(s)\Big) \\
&= \Big(A(t) - B(t)B(t)\t\Pi(t)\Big) 
\\* 
&\qquad\qquad \times \Big(\Phi_{11}(t, s) + \Phi_{12}(t, s) \Pi(s)\Big).
\end{align*}
This completes the proof. \hfill
\end{proof}


\begin{proof}[Proof of Lemma \ref{lem:map-inv-Q}] 
Clearly, $\partial \bar{f}\big(\vect(\Pi_{0})\big)$ is continuous in $\vect(\Pi_{0})$. 
First, we show that $\partial \bar{f}\big(\vect(\Pi_{0})\big)$ is nonsingular at each $\Pi_{0} \prec - \big(\Phi_{12}^{10}\big)^{-1} \Phi_{11}^{10}$. 
Since $\Phi_{\Pi}^{10} \otimes \Phi_{\Pi}^{10}$ is nonsingular, it suffices to show that the term in the square brackets of \eqref{eqn:map-jacob-Q}, that is,
\begin{multline*}
S \triangleq \Sigma_{0} \otimes W_{10} + W_{10} \otimes \Sigma_{0} 
\\* 
+ \int_{0}^{1} P_{s} \otimes \big(W_{10} - W_{s0}\big) + \big(W_{10} - W_{s0}\big) \otimes P_{s} \, \dd s,
\end{multline*}
is nonsingular. 
One can readily check that $S$ is symmetric, because $\Sigma_{0} \succ 0$, $W_{10} \prec 0$, $P_{s} \succeq 0$, and $W_{10} - W_{s0} \preceq 0$ are all symmetric. 
Let $X \neq 0$ be an $n \times n$ matrix. 
Then,
\begin{align*}
&\vect(X)\t S \vect(X) = \vect(X)\t \vect\bigg( W_{10} X \Sigma_{0} + \Sigma_{0} X W_{10} 
\\* 
&\hspace{13mm} + \int_{0}^{1} \big(W_{10} - W_{s0}\big) X P_{s} + P_{s} X \big(W_{10} - W_{s0}\big) \, \dd s \bigg) \\
&= \trace \bigg( X\t W_{10} X \Sigma_{0} + X\t \Sigma_{0} X W_{10} 
\\* 
&\hspace{4mm} + \int_{0}^{1} X\t \big(W_{10} - W_{s0}\big) X P_{s} + X\t P_{s} X \big(W_{10} - W_{s0}\big) \, \dd s \bigg) \\
&\leq \trace \Big( \Sigma_{0}^{\frac{1}{2}} X\t W_{10} X \Sigma_{0}^{\frac{1}{2}} + \Sigma_{0}^{\frac{1}{2}} X W_{10} X\t \Sigma_{0}^{\frac{1}{2}} \Big) < 0.
\end{align*}
Thus, $S \prec 0$. 
Therefore, $\partial \bar{f}\big(\vect(\Pi_{0})\big)$ is nonsingular at each $\vect(\Pi_{0})$ in the domain of $\bar{f}$.

Next, we show that the map $f$ is proper, that is, for any compact subset $\mathcal{K} \subset \{\Sigma_{1} \in \mathbb{R}^{n \times n}: \, \Sigma_{1} = \Sigma_{1}\t \succ 0\}$, the inverse image $f^{-1}(\mathcal{K}) \subset \{\Pi_{0} \in \mathbb{R}^{n \times n}: \, \Pi_{0} = \Pi_{0}\t \prec - \big(\Phi_{12}^{10}\big)^{-1} \Phi_{11}^{10}\}$ is compact. 
Since $f$ is continuous and $\mathcal{K}$ is closed, the inverse image $f^{-1}(\mathcal{K})$ is also closed. 
Since $\mathcal{K}$ is bounded in $\mathbb{R}^{n \times n}$, in view of 
\eqref{map:pi0-sigma1-Q}, the set 
\begin{equation*}
\Big\{ \Big(\Phi_{11}^{10} + \Phi_{12}^{10} \Pi_{0}\Big) \Sigma_{0} \Big(\big(\Phi_{11}^{10}\big)\t + \Pi_{0} \big(\Phi_{12}^{10}\big)\t \Big): \, \Pi_{0} \in f^{-1}(\mathcal{K}) \Big\}
\end{equation*}
is also bounded in $\mathbb{R}^{n \times n}$. 
Since $\Sigma_{0}$ and $\Phi_{12}^{10}$ are invertible, $f^{-1}(\mathcal{K})$ is bounded in $\mathbb{R}^{n \times n}$. 
Therefore, $f^{-1}(\mathcal{K})$ is compact, and thus $f$ is proper. 
Since the set of positive definite matrices is convex, it is simply connected~\cite{krantz2002implicit}. 
By Hadamard's global inverse function theorem~\cite{krantz2002implicit}, $f$ is a homeomorphism. \hfill
\end{proof}


In light of \eqref{eqn:phi-AB-pi-Q}, for notational simplicity, let 
\begin{align*}
\Phi_{\Pi}(t, s) &\triangleq \Phi_{A-BR^{-1}B\t\Pi}(t, s) = \Phi_{11}(t, s) + \Phi_{12}(t, s) \Pi(s), 
\\
\bar{M}(t, s) &\triangleq \int_{s}^{t} \Phi_{\Pi}(t, \tau) B(\tau) R(\tau)^{-1} B(\tau)\t \Phi_{\Pi}(t, \tau)\t \, \dd \tau. 
\end{align*}

\begin{lemma} \label{lem:gramian-AB-pi-Q}
For all $t, s \in \mathbb{R}$, 
\begin{equation} \label{eqn:gramian-AB-pi-Q}
\bar{M}(t, s) = - \Phi_{12}(t, s) \Phi_{\Pi}(t, s)\t. 
\end{equation}
\end{lemma}

\begin{proof} 
When $t = s$, $\bar{M}(s, s) = 0 = - \Phi_{12}(s, s) \Phi_{\Pi}(s, s)\t$. 
It is straightforward to check that $\bar{M}(t, s)$ satisfies the linear differential equation 
\begin{align*}
\frac{\partial}{\partial t} \bar{M}(t, s) 
&= \Big(A(t) - B(t) R(t)^{-1} B(t)\t \Pi(t)\Big) \bar{M}(t, s) 
\\* 
&\hspace{6mm}
+ \bar{M}(t, s) \Big(A(t) - B(t) R(t)^{-1} B(t)\t \Pi(t)\Big)\t 
\\* 
&\hspace{6mm}
+ B(t) R(t)^{-1} B(t)\t. 
\end{align*}
Next, we verify that $- \Phi_{12}(t, s) \Phi_{\Pi}(t, s)\t$ satisfies the same linear differential equation. 
\begin{align*}
&\frac{\partial}{\partial t} \Big(- \Phi_{12}(t, s) \Phi_{\Pi}(t, s)\t\Big) 
\\ 
&= - \Big(\frac{\partial}{\partial t} \Phi_{12}(t, s)\Big) \Phi_{\Pi}(t, s)\t - \Phi_{12}(t, s) \Big(\frac{\partial}{\partial t} \Phi_{\Pi}(t, s)\t\Big) 
\\
&= - \Big(A(t) \Phi_{12}(t, s) - B(t) R(t)^{-1} B(t)\t \Phi_{22}(t, s)\Big) \Phi_{\Pi}(t, s)\t 
\\* 
&\hspace{5mm} 
- \Phi_{12}(t, s) \Phi_{\Pi}(t, s)\t \Big(A(t) - B(t) R(t)^{-1} B(t)\t \Pi(t)\Big)\t 
\\
&= \Big(- \Phi_{12}(t, s) \Phi_{\Pi}(t, s)\t\Big) \Big(A(t) - B(t) R(t)^{-1} B(t)\t \Pi(t)\Big)\t 
\\* 
&\hspace{5mm} 
+ A(t) \Big(- \Phi_{12}(t, s) \Phi_{\Pi}(t, s)\t\Big) 
\\* 
&\hspace{5mm} 
+ B(t) R(t)^{-1} B(t)\t \Phi_{22}(t, s) \Phi_{\Pi}(t, s)\t. 
\end{align*}
By \eqref{eqn:phi-AB-pi-Q}, we have $\Big(\Phi_{11}(s, t)\t + \Pi(t) \Phi_{12}(s, t)\t\Big) \Phi_{\Pi}(t, s)\t = I_{n}$. 
By \eqref{eqn:phi-11-22} and \eqref{eqn:phi-12-equiv}, $\Big(\Phi_{22}(t, s) - \Pi(t) \Phi_{12}(t, s)\Big) \Phi_{\Pi}(t, s)\t = I_{n}$. 
It follows that 
\begin{align*}
&\frac{\partial}{\partial t} \Big(- \Phi_{12}(t, s) \Phi_{\Pi}(t, s)\t\Big) 
\\ 
&= \Big(A(t) - B(t) R(t)^{-1} B(t)\t \Pi(t)\Big) \Big(- \Phi_{12}(t, s) \Phi_{\Pi}(t, s)\t\Big) 
\\* 
&\hspace{5mm}
+ \hspace{-0.5mm} \Big( \hspace{-1mm} - \Phi_{12}(t, s) \Phi_{\Pi}(t, s)\t \hspace{-0.5mm} \Big) \hspace{-0.5mm} \Big( \hspace{-0.5mm} A(t) - B(t) R(t)^{-1} B(t)\t \Pi(t) \hspace{-0.5mm} \Big)\t 
\\* 
&\hspace{5mm}
+ B(t) R(t)^{-1} B(t)\t. 
\end{align*}
Since $\bar{M}(t, s)$ and $- \Phi_{12}(t, s) \Phi_{\Pi}(t, s)\t$ satisfy the same linear equation with the same initial condition, \eqref{eqn:gramian-AB-pi-Q} holds. \hfill
\end{proof}

\end{document}